\documentclass[%
aip,
amsmath,amssymb,
reprint,%
]{revtex4-1}

\usepackage{graphicx}
\usepackage{xcolor}
\usepackage{tikz}
\usepackage{float}
\usepackage{subfigure}
\usepackage{natbib}
\usepackage{dcolumn}
\usepackage{bm}
\usepackage{amsthm}
\usepackage[utf8]{inputenc}
\usepackage[T1]{fontenc}
\usepackage{mathptmx}
\usepackage{etoolbox}
\newtheorem{theorem}{Theorem}
\newtheorem{lemma}[theorem]{Lemma}
\newtheorem{corollary}[theorem]{Corollary}

\newtheorem{definition}{Definition}

\makeatletter
\def\@email#1#2{%
	\endgroup
	\patchcmd{\titleblock@produce}
	{\frontmatter@RRAPformat}
	{\frontmatter@RRAPformat{\produce@RRAP{*#1\href{mailto:#2}{#2}}}\frontmatter@RRAPformat}
	{}{}
}%
\makeatother
\begin{document}
	
	\preprint{AIP/123-QED}
	
	\title{Intermittent--synchronization in non-weakly coupled piecewise linear expanding map lattice: a geometric-combinatorics method}
	\author{Junke Zhang}
%
	\affiliation{
		Department of Mathematics,
		Nanjing University, Nanjing, 210008.   xiaopangzhang223@163.com
	}%
		\author{Yiqian Wang}
	
	\affiliation{%
			Department of Mathematics,
		Nanjing University, Nanjing, 210008.   Corresponding author: yiqianw@nju.edu.cn
	}%
	
	\date{\today}
	\begin{abstract}
		The coupled (chaotic) map lattices (CMLs) characterizes the collective dynamics of a spatially distributed system consisting of locally or globally coupled maps.
		 The current research on the dynamic behavior of CMLs is based on the framework of the Perron-Frobenius operator and mainly focuses on weakly-coupled cases. In this paper,  a novel geometric-combinatorics method for for non weakly-coupled CMLs is provided on the dynamical behavior of a two-node CMLs with identical piecewise linear expanding maps.
		We obtain a necessary-sufficient condition for  the uniqueness of absolutely continuous invariant measure (ACIM) and for the occurrence of intermittent-synchronization, that is, almost each orbit enters and exits an arbitrarily small neighborhood of the diagonal for an infinite number of times.
		
	\end{abstract}
		\maketitle


	
	\section{\label{sec:1}Introduction}
	Coupled map lattices (CMLs) were first proposed  by Kaneko et al.\cite{kaneko1984period,kaneko1985spatial, waller1984spatial,kapral1985pattern,kuznetsov1983critical,kuznetsov1984model}.
	CML describe the temporal evolution of fields that can be expressed as the independent evolution of local subsystems (or elements) within these fields, typically governed by a map defined on a local phase space. This is followed by spatial interactions among these local subsystems, which are determined by an operator that acts on the global phase space of the CML.
	The local subsystems in a CML are organized in a lattice structure with the assumptions that all local dynamical systems are identical and that the spatial interactions between any given local system and the rest of the lattice are uniform across all subsystems.
	
	Chaotic synchronization  can be observed in coupled physical systems  \cite{lin2002chaotic,heagy1994chaotic} provided  the strength of this coupling exceeds a certain critical threshold,
   while intermittent chaotic synchronization \cite{pecora2014cluster}  can be observed in  CMLs with a coupling smaller than the critical threshold.
	Intermittent-synchronization is the alternation dynamical behavior of (almost) synchronized and desynchronized states for infinitely many times, a special case of  cluster synchronization (CS), in which
	sets of synchronized elements emerge \cite{pogromsky2008partial}.

	This phenomenon manifests across scales such as clustering or pendulum waves in weakly coupled pendulum clocks by Huygens \cite{czolczynski2009clustering}; and in neurodynamics -- intermittent synchronous firings of neurons is the mechanism responsible for Epilepsy, Parkinson's and other neurological disorders, and  plays a very important role in the brain \cite{takahashi2010circuit,bullmore2009complex,kitajima2009cluster,ochi2008synchronization}.
Two-State on-off Intermittency and the onset of Turbulence was discussed in \cite{galuzio2010two}.
In \cite{boccaletti2000characterization} intermittent lag synchronization of two nonidentical symmetrically coupled R\"ossler systems is investigated. The effect of noncoherence on the onset of phase synchronization of two coupled chaotic oscillators was studied in \cite{Osipov2003three}. In \cite{budzinski2017detection}	the nonstationary transition to synchronized states of a neural network was displayed by a complex neural network. For more details, we can refer to Ding et al.\cite{ding1997stability,choudhary2017small,Papo2024brain}.
	These phenomena underscore the need for analytical frameworks to dissect the mechanisms of intermittent-synchronization.
	The central theoretical challenge lies in identifying universal mechanisms governing intermittent-synchronization.
	A more difficult question is to explore the mechanism on the occurrence of  transition between synchronization and intermittent-synchronization.
	However, theoretical results on these questions are very rare.
	This motivates this paper.

\subsection{Coupled map lattice model}
	
	Let $f : [0,1] \to [0,1]$ be a piecewise expanding map,
	$I$ be the $d \times d$ identity matrix, and
	$A$ be a $d \times d$ matrix such that $A \bm{e} = 0$, where $\bm{e}= (1, 1, \cdots, 1)^T$.
	For $d\ge 2$, the matrix $A$ is called the coupling matrix. The following coupling matrix corresponds to the nearest neighbor coupling (or Laplacian coupling):
	\[
	\begin{bmatrix}
		-2 & 1 & 0 & \cdots & 0 & 1 \\
		1 & -2 & 1 & \cdots & 0 & 0 \\
		0 & 1 & -2 & \cdots & 0 & 0 \\
		\vdots & \vdots & \vdots & \ddots & \vdots & \vdots \\
		0 & 0 & 0 & \cdots & -2 & 1 \\
		1 & 0 & 0 & \cdots & 1 & -2
	\end{bmatrix}_{\!\!d \times d},
	\]
	while
	the global coupling CML corresponds to
	\[
	\begin{bmatrix}
		-(d-1) & 1 & 1 & \cdots & 1 \\
		1 & -(d-1) & 1 & \cdots & 1 \\
		1 & 1 & -(d-1) & \cdots & 1 \\
		\vdots & \vdots & \vdots & \ddots & \vdots \\
		1 & 1 & 1 & \cdots & -(d-1)
	\end{bmatrix}_{\!\!d \times d}.
	\]
	When $d=2$, the coupling matrix is
	\[
	A = \left( \begin{array}{cc}
		-1 & 1 \\
		1 & -1
	\end{array} \right).
	\]
	Then a CML is a dynamical system generalized by the following map on $[0,1]^d$:
	\begin{equation}\label{equation:1}
		T: \ \bm{x} (n+1) = (I + cA) \bm{f}(\bm{x}(n)),
	\end{equation}
	where $c\in [0, \frac{1}{2}]$ ( $c\in [0,1]$ for $d=2$ ) for the nearest neighbor coupling (or $c \in[0,\frac{1}{d-1}]$ for the global coupling) is the coupling strength, $\bm{x}(n) = (x_1(n),\cdots, x_d(n))^T \in [0, 1]^d$ for $n \in \mathbb{N} \cup \{0\}$, and $\bm{f}(\bm{x}(n)) = (f(x_1(n)),\cdots, f(x_2(n)))^T$.
	
	Since $A \bm{e} = 0$, it is evident that the diagonal set $S_{inv} = \{(x_1,\cdots, x_d) \in [0, 1]^d \mid x_1 = x_2=\cdots=x_d\}$ is an invariant set of synchronized points under the iteration of $T$.
	Then that $S_{inv}$ of (\ref{equation:1}) is a global attractor is  equivalent to that synchronization occurs in CML (\ref{equation:1}).
	Moreover, intermittent-synchronization is equivalent to the following	
	\begin{eqnarray} \label{equation:2}
		\liminf_{n \to \infty} \, dist(\bm{x}(n), S_{inv}) = 0,\\
		\limsup_{n \to \infty} \, dist(\bm{x}(n), S_{inv}) \geq r_0 > 0.\label{equation:3}
	\end{eqnarray}
	for some $r_0>0$ and almost all initial points, where $dist(A, B)$ denotes the distance between two points or sets.
	In other words, the successive alternation between the ordered phase (almost synchronization) described by (\ref{equation:2}) and the disordered phase (desynchronization) described by (\ref{equation:3}) occurs for almost every orbit.
	
	When the number of nodes is two, (\ref{equation:2}) and (\ref{equation:3}) are equivalent to the following equation and inequality, respectively:
\begin{eqnarray}\label{equation:4}
	\liminf_{n \to \infty} \, |x_1(n) - x_2(n)| = 0,\\
	\limsup_{n \to \infty} \, |x_1(n) - x_2(n)| \geq r_0 > 0.\label{equation:5}
\end{eqnarray}
The intermittent-synchronization of 2-node CML is shown in the panel (a) of Fig. \ref{Fig.2node1} while the synchronization of 2-node CML is shown in the panel (b) of Fig. \ref{Fig.2node1}.
\begin{figure}
	\includegraphics[width=0.48\textwidth]{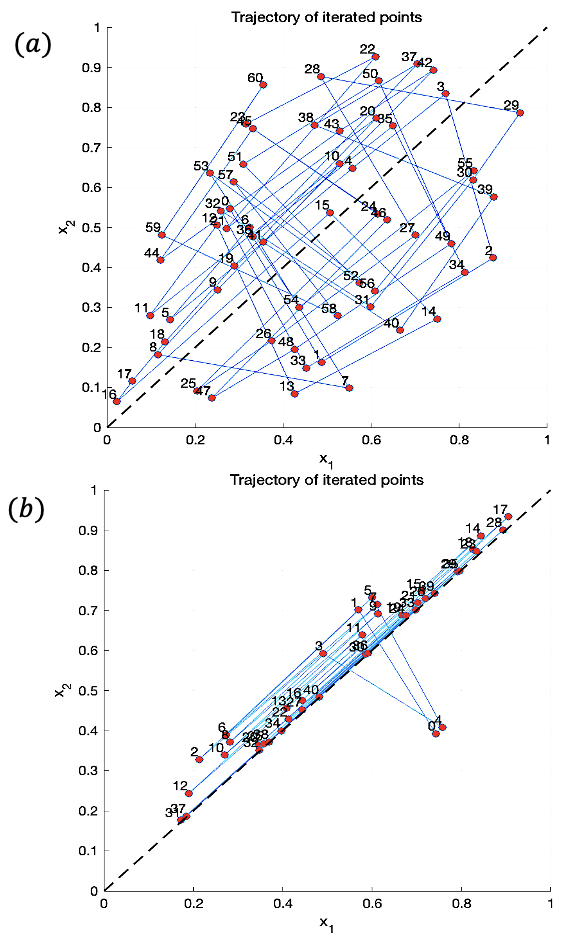}
	\caption{\label{Fig.2node1}The figures show the trajectory of any point in the phase space $[0,1]^{2}$ after iterations in the 2-node coupled Lorenz map lattice.
		The number above each red dot represents the number of iterations corresponding to the point.
		In panel (a), $c=0.15$ which shows intermittent-synchronization; and in panel (b), $c=0.28$ which shows synchronization.}
\end{figure}

Li et al.  \cite{li2021intermittent} proved that for the coupled map lattice (CML) with $d=2$ and $f$ being the standard tent map
	\begin{equation*}
		f(x) = \left\{
		\begin{array}{ll}
			2x, & 0 \leq x < \frac{1}{2}, \\
			2 - 2x, & \frac{1}{2} \leq x \leq 1,
		\end{array}
		\right.
	\end{equation*}
synchronization--intermittent synchronization  transition occurs.
More precisely, for almost every initial point in the phase space $[0, 1]^2$, intermittent-synchronization occurs  for any $c \in [0, \frac{1}{4}) \cup (\frac{3}{4}, 1]$ and for $c \in [\frac{1}{4}, \frac{3}{4}]$ synchronization occurs.
	This implies that $c = \frac{1}{4}$ (or $\frac{3}{4}$) is the critical threshold.

 While previous works depend on the analysis of Perron-Frobenius operator, the new method \cite{li2021intermittent} is a combination of geometry and combinatorics. Roughly speaking, it is based on the observation that the occurrence of synchronization or intermittent-synchronization depends on a competition of the expansibility of maps and the complexity of the dynamical system due to the fact that $f$ is not one-to-one. When the coupling is weaker than the critical coupling $\frac14$, Li et al. proved that the factor from the expansibility of maps is stronger and thus obtained a mixing property of CML, which implies intermittent-synchronization; otherwise, synchronization is obtained.

 It is worthy to say more on this method. Since the map is piecewise expanding, locally a small domain is mapped to another one with a larger Lebesgue measure.	On the other hand,  the map is not one-to-one. Thus one cannot easily tell whether or not a domain with an arbitrary small positive measure will grow into one with a measure large enough.
	In other words, for the growth of the measure of a domain under the repeated actions of a map, there is a competition between the expansibility of the map and the partition of system.
	If one can show that averagely the factor from the expansibility of the map surpasses the factor from  the non-injection of the map, then we can obtain a conclusion similar to topological mixing.

	In this paper, we aim to extend the synchronization--intermittent synchronization transition result to the following CML
	\begin{eqnarray}\label{model}
		T: \left\{
		\begin{array}{ll}
			x_1(n+1) = (1-c)f(x_1(n)) + c f(x_2(n)), \\
			x_2(n+1) = c f(x_1(n)) + (1-c) f(x_2(n)),
		\end{array}
		\right.
	\end{eqnarray}
	with piecewise linear expanding maps
	\begin{eqnarray}\label{f1f2}
	f(x)= 2x~ \text{mod} ~ 1 ~~~\text{or}~~~~ f(x)=\pm 3x ~ \text{mod} ~  1.
	\end{eqnarray}

	We will also adopt the geometry-combinatorics method \cite{li2021intermittent}. However, due to the discontinuity of the Lorenz-type maps, the corresponding CMLs are more difficult and the original method \cite{li2021intermittent} has to be modified significantly. Similarly, for CMLS with other piecewise linear expanding maps, the concrete difficulties are quite different.

	It is worthy to point out that Jost and Joy \cite{jost2001spectral} explored the necessary and sufficient condition for the occurrence of the local synchronization.
	Our results showed that for many piecewise linear expanding maps on a compact domain, local synchronization is equivalent to global synchronization (see next section on the analysis of the Lyapunov exponent on the diagonal).

	Moreover, Tsujii \cite{tsujii2001absolutely} obtained  the existence of finitely many ACIMs for any piecewise linear expanding map on a compact domain, which is applicable for our two models in the intermittent-synchronization situation.
	Another purpose of this paper is to prove the uniqueness of ACIM with the mixing property of the corresponding CMLs based on the proof of intermittent-synchronization and the existence of finitely many ACIMs \cite{tsujii2001absolutely}.


	\subsection{The motivation}
	This paper aims to show that a novel geometric-combinatorics method, together with the traditional transfer operator method (Perron-Frobenius operator), can yield a sharp result on the dynamical behavior of CMLs.

There are a series of celebrated works on the ergodic properties of CMLs with the method of Perron-Frobenius operator. In the weakly-coupling situation, Bunimovich et al. \cite{bunimovich1988spacetime,keller1997mixing, keller1996coupled, keller1992transfer, saussol2000absolutely,liverani2013multidimensional,keller2006uniqueness} and references therein obtained the existence of finitely many ACIMs in CMLs with piecewise expanding maps. It is worthy to note that these results depend heavily on the smallness condition  on the coupling strength. Moreover, Keller et al.\cite{keller1992some} also  proved the uniqueness of ACIM and synchronization--intermittent-synchronization transition in CML with $d=2$, $f$ standard (slope=$2$) Lorenz type map and the critical threshold $\frac14$ and their method seems difficult to be applied directly on $3$ or more nodes. In addition, Keller \cite{keller1997mixing} obtained the uniqueness of ACIM in CML with tent map and weak couplings (independent on the size) and the smallness of coupling seems to be necessary in his proof.
 Tsujii \cite{tsujii2001absolutely} showed the existence of finitely many ACIMs for general expanding piecewise linear maps. However, these are far from a kind of sufficient-necessary result which was supported by our numerical simulations. In fact, even expanding condition may not be necessary for the existence and uniqueness of ACIM in CMLs with piecewise linear maps. To bridge this gap, a fresh research paradigm seems to be necessary.
	
	It is well known that the CML with identical maps is symmetric with respect to the synchronization manifold (the diagonal). It implies that the diagonal must play a key role in the dynamical behavior of CMLs, which is beyond the consideration of abstract functional sketch\cite{bunimovich1988spacetime,keller1997mixing, keller1996coupled,keller1992some, keller1992transfer, saussol2000absolutely,liverani2013multidimensional,keller2006uniqueness}.. In fact, the information from the diagonal is necessary in the sufficient-necessary result of Keller et al \cite{keller1992some}. This implies the importance of systemically developing a  geometric method.
	In this paper, we develop the geometric-combinatorics method originated in \cite{li2021intermittent} to study the role of the diagonal on  CMLs with the map as in (\ref{f1f2}), where the discontinuity of the map complicates the study. Consequently we obtain the existence of a threshold on the coupling strength for the occurrence of intermittent-synchronization for CML, which coincides with numerical results. The occurrence of intermittent-synchronization then help us obtain some kind of mixing property, which together with the existence of finitely many ACIMs \cite{tsujii2001absolutely}, implies the uniqueness of ACIM. It together with
Tsujii's result \cite{tsujii2001absolutely} shows that the trajectory of almost each point in the sense of Lebesgue measure will visit almost everywhere of a region with a positive Lebesgue measure for infinite times.

Moreover,  the analysis of the Lyapunov exponent on the diagonal\cite{jost2001spectral} (see next section) also shows the occurrence of local synchronization-asynchronization transition. Combining this with our intermittent synchronization result,  in this paper we actually obtained a sufficient-necessary result for the occurrence of intermittent synchronization and the uniqueness of ACIM since local synchronization implies the nonexistence of ACIM and intermittent synchronization.  Indeed, all our other works based on the geometric-combinatorics method are of sufficient-necessary results as follows. More precisely, we obtained the synchronization--intermittent synchronization transition and the uniqueness of ACIM for the following different situations: coupled nonstandard tent maps (the absolute value of the slope is less that $2$) lattice with $2$ nodes\cite{wang2025synchronization},  coupled standard tent map lattice  with $3$ nodes \cite{wang2025threenode}, and global-coupled standard tent map lattices with any finitely many nodes \cite{wang2025threenode}. It is worthy to note that the details in the results mentioned above are much more complicated than in \cite{li2021intermittent} where the size is $2$ and the absolute value of the slope is $2$.
In comparison, there are very few theoretical results on the existence of a critical threshold for dynamical behaviors of CMLs. Two related results can be found by Keller for coupled Lorenz type map lattices with $2$ nodes, $f$ standard (slope=$2$) Lorenz type map \cite{keller1992some} and Maere \cite{de2010phase}.

	Furthermore, intermittent-synchronization may provide a connection between the uncoupled single dynamical system and the dynamical system of CML. In fact, intermittent-synchronization means that almost every orbit can stay in any neighbor of the diagonal and thus shadows a typical orbit on the diagonal for any finitely long time (which implies that the diagonal is in the support of any ACIM). However, the system on the diagonal is equivalent to the product of uncoupled single dynamical systems.  Hence it seems reasonable to consider the problem that how much information on the dynamical behavior of CMLs can be obtained from the uncoupled single dynamical system. For example, both the uncoupled single system (\ref{f1f2}) and their CML counterparts are topological mixing. It suggests some similarity between them.

	Note that for coupled tent map lattices, the proofs for $2$ and $3$ nodes are quite different. This is due to the fact that the situation for different size may be different. In fact, in the case of nearest-neighbour coupling tent map lattices, the necessary condition for the occurrence of the synchronization- intermittent synchronization transition is the number of nodes is less than $6$ for the following reason.
One can easily check that (local) synchronization cannot occur with the formula on the eigenvalue of a Laplacian matrix if the size is larger than $6$, since on the diagonal line (the synchronized manifold), there always exists at least one eigenvector orthogonal to the diagonal line with an eigenvalue larger than $1$.
  This implies that the proof may depends on the size and new difficulties may exist for larger size. In fact, the proof for coupled tent map lattices with $3$ nodes is more complicated than the one for $2$ nodes.

  Potentially our method may be applicable in the study of nearest neighbour coupling map lattices with any finite many nodes.
  These  are considered in one of our ongoing work.
 For the $n$-node and non-weakly coupling case, the key point is to control the growth speed of the number of convex subdomains since there are more discontinuity spaces in higher dimension than in $2$ dimension.
 The success in coupled Lorenz map lattice with $2$ nodes and coupled tent map lattice with $3$ nodes lies in the existence sufficiently many invariant hyperplanes by which we can reduce the problem to lower dimensions.
 n general cases, however there are no sufficiently many invariant hyperplanes. Thus we need to find other (quasi)invariants to control the growth speed of  $\#\{\text{convex subdomains}\}$.

 Another future objective of us is to generalize the geometric-combinatorics method to maps with a nonlinearity much stronger than mentioned above and thus show that uniformly expanding is not necessary for the occurrence of intermittent synchronization and the existence of ACIM. Currently, some partial result has been obtained.
	
Hopefully, we can also consider other coupling topologies besides the nearest-neighbor coupling with our method, say long-range coupling, global coupling, etc, which can be found frequently in other models, e.g., discrete Schr\"odinger operators\cite{Damanik2022operator}.

	\subsection{The structure of the paper}
In Sec. \ref{Sec:main results}, we elaborate on our two main theorems.
In Sec. \ref{sec:2}, we introduce the iterative lemma, which serves as a prerequisite for analyzing intermittent-synchronization in CML using the geometric-combinatorics method.
In Sec. \ref{sec:3} and Sec. \ref{sec:4}, we proof the phenomenon of intermittent-synchronization  and the uniqueness of ACIM in  two-node CML with two kinds of piecewise linear expanding maps respectively.
Using the geometric-combinatorics method, we  prove the intermittent-synchronization result rigorously and obtain the complete range of coupling strengths for which intermittent-synchronization occurs in CML, then we get the uniqueness of absolutely continuous invariant measure for our model (\ref{model}).
Notably, our geometric-combinatorics method exhibits strong applicability to multi-node CMLs and other coupled piecewise expanding maps.
Therefore, in Sec. \ref{sec:5}, we extend this approach to various CMLs and compare it with traditional Forbenius-Perron operator-based methods.
Section \ref{sec:6} concludes.

\section{Main results}\label{Sec:main results}

In this paper, we primarily present and prove the following two theorems.

\begin{theorem}\label{Th1}
	Consider the coupled map lattice
	\begin{eqnarray*}
		T: \left\{
		\begin{array}{ll}
			x_1(n+1) = (1-c)f(x_1(n)) + c f(x_2(n)), \\
			x_2(n+1) = c f(x_1(n)) + (1-c) f(x_2(n)),
		\end{array}
		\right.
	\end{eqnarray*}
	with standard Lorenz map
	\[
	f(x) = \left\{
	\begin{array}{ll}
		2x, & 0 \leq x < \frac{1}{2}, \\
		2x - 1, & \frac{1}{2} \leq x \leq 1.
	\end{array}
	\right.
	\]
	Then $c=\frac{1}{4}$ is the synchronization--intermittent synchronization transition point, that is,
	for all $c \in [0, \frac{1}{4}) \cup (\frac{3}{4}, 1]$, there exists a constant $r_0 > 0$ such that for almost every initial point $(x_1(0), x_2(0)) \in [0, 1]^2$, the following hold:
	\begin{eqnarray*}
		\left\{
		\begin{array}{ll}
			\liminf_{n \to \infty} \, dist(\bm{x}(n), S_{\text{inv}}) = 0 , \\
			\limsup_{n \to \infty} \, dist(\bm{x}(n), S_{\text{inv}}) \geq r_0 > 0 ;
		\end{array}
		\right.
	\end{eqnarray*}
	and for each $\frac{1}{4} < c < \frac{3}{4}$,
	\[
	\lim_{n \to \infty} \, dist(\bm{x}(n), S_{\text{inv}}) = 0
	\]
	holds for  every initial point.
	
	Moreover, when $c \in [0, \frac{1}{4}) \cup (\frac{3}{4}, 1]$, there exist a unique absolutely continuous invariant measure for $T$.
\end{theorem}

The panel (a) of Fig.  \ref{Fig.2node1}  and the red line of Fig.  \ref{Fig.2node2}  illustrate the intermittent-synchronization phenomenon of the coupled Lorenz map lattice from two perspectives,
while the panel (b) of Fig.  \ref{Fig.2node1} and the blue line of Fig. \ref{Fig.2node2}  depict the synchronization phenomenon of the coupled Lorenz map lattice.

Note that when $\frac{1}{4} < c < \frac{3}{4}$, Keller\cite{keller1992some} has concluded that the diagonal attracts Lebesgue-a.e. orbit. Thus, we can deduce that synchronization occurs for each $\frac{1}{4} < c < \frac{3}{4}$, that is,
\[
\lim_{n \to \infty} \, dist(\bm{x}(n), S_{inv}) = 0
\]
for every initial point.
Therefore, according to the theorem we obtained, we can conclude that $c = \frac{1}{4}$ ($c = \frac{3}{4}$ symmetrically) is the synchronization--intermittent synchronization transition point.

\begin{figure}
	\includegraphics[width=0.49\textwidth]{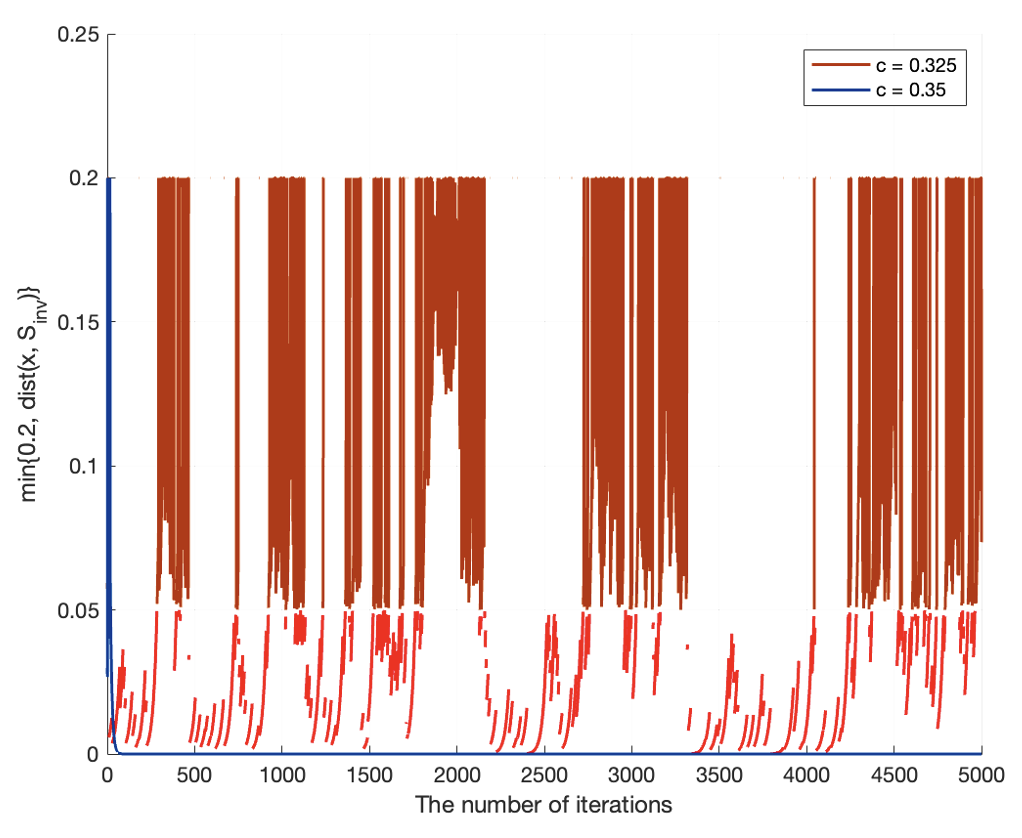}
	\caption{\label{Fig.2node2} The figure shows the distance of any point in the phase space $[0,1]^{2}$ to the invariant set $S_{inv}=\{\bm{x}\in[0,1]^{2}| x_{1}=x_{2}\}$  in the 2-node coupled map lattice with $f(x)=3x \mod 1$.
		We only show the part with the distance less than 0.2 in this figure.
		The red line corresponds to $c=0.325$ which shows intermittent-synchronization, we distinguish distances greater than $0.05$ and less than $0.05$ using different red colors; and the blue line corresponds to $c=0.35$ which shows synchronization.}
\end{figure}

We demonstrate that our synchronization--intermittent synchronization transition conclusion coincides with the transition of  the Lyapunov exponents \cite{jost2001spectral} of the diagonal invariant set \( S_{inv} \).
The points on the diagonal satisfy \( x_1 = x_2 \), so the iterations of this CML along the diagonal is governed by
\[
x_{n+1} = T(x_n, x_n) = f(x_n),
\]
which depends solely on the Lorenz map  \( f(x) \).
The Lyapunov exponent along the diagonal is given by
\[
\lambda_{\parallel} = \lim_{N \to \infty} \frac{1}{N} \sum_{n=0}^{N-1} \ln | f'(x_n) |.
\]
Since the slope of the Lorenz map is constantly 2,  we obtain
\[
\lambda_{\parallel} = \ln 2.
\]

To determine the Lyapunov exponent in the transverse direction, we consider a small perturbation \( \delta_n = x_1 - x_2 \) and by using the iterative equations of CML, we approximate
\[
\delta_{n+1} = ((1-c) - c) f'(x_n) \delta_n = (1-2c) f'(x_n) \delta_n.
\]
Thus, the Lyapunov exponent in the transverse direction is
\[
\lambda_{\perp} = \lim_{N \to \infty} \frac{1}{N} \sum_{n=0}^{N-1} \ln |(1-2c) f'(x_n)|= \ln |2(1-2c) |.
\]

Therefore, the Lyapunov exponents of~$S_{inv}$~are
\[
\lambda_{\parallel} = \ln 2, \quad \lambda_{\perp} = \ln |2(1-2c) |.
\]
For \( c \in (0, \frac{1}{4}) \),  \( \lambda_{\perp} = \ln |2(1-2c) | > 0 \), perturbations in the transverse direction grow exponentially, indicating that the trajectory on the diagonal is transversely unstable.
In other words, even if the initial point is arbitrarily close to the synchronized state, small random perturbations will cause it to rapidly diverge from the synchronization manifold.
In contract, for  $c\in (\frac{1}{4},\frac{3}{4})$,  \( \lambda_{\perp} = \ln |2(1-2c) | <0 \), any  point through iterations will gradually get infinitely closer to $S_{inv}$ and never move away from it, indicating that synchronization occurs for $c\in (\frac{1}{4},\frac{3}{4})$.
Consequently,~$c=\frac{1}{4}$~is the critical value between these two different phenomena.

\begin{figure}
	\includegraphics[width=0.49\textwidth]{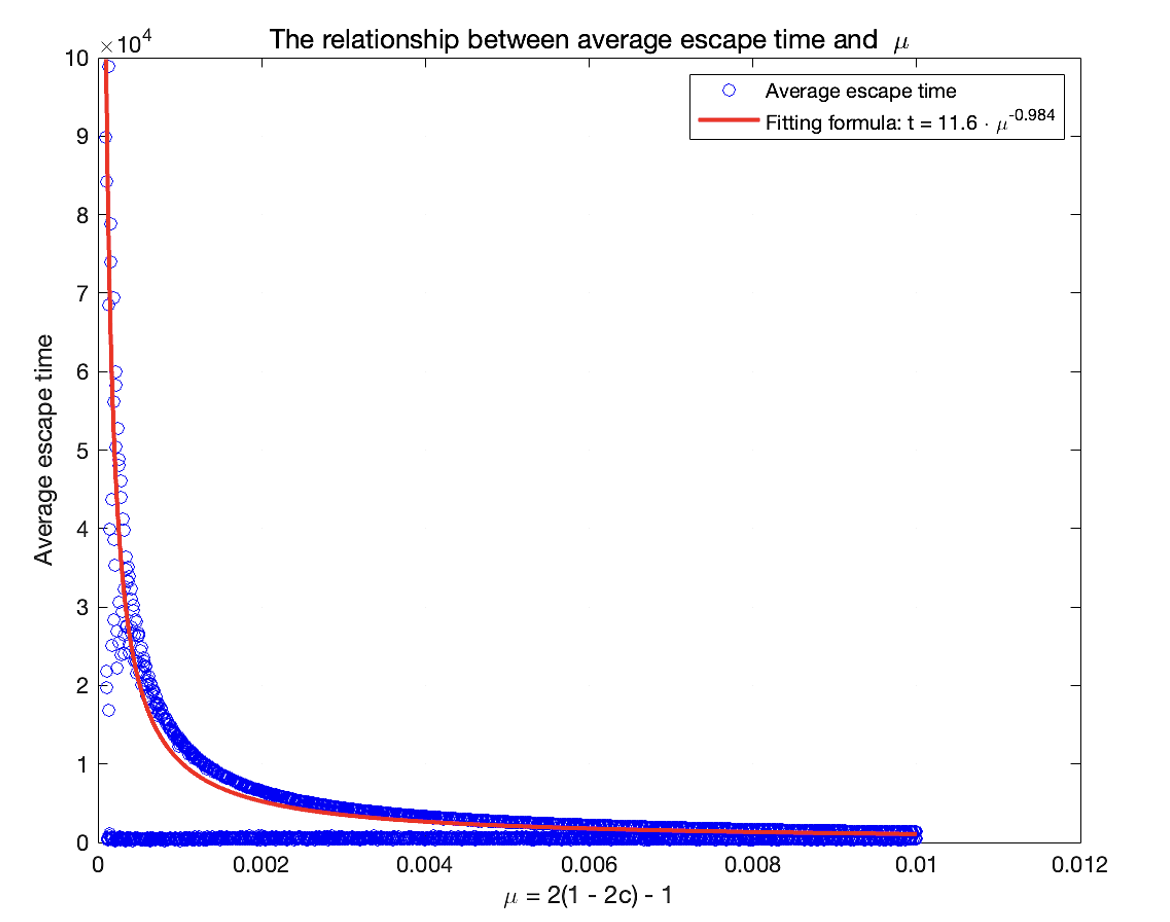}
	\caption{\label{Fig.escapetime} The figure shows that  the average escape time from any point in the $10^{-12}$-neighborhood of the diagonal $S_{inv}$ to outside the $10^{-6}$-neighborhood of $S_{inv}$ is approximately inversely proportional to the distance between the coupling strength c and $1/4$.}
\end{figure}

Through numerical simulation, we find that the average escape time from any point in the $10^{-12}$-neighborhood of the diagonal $S_{inv}$ to outside the $10^{-6}$-neighborhood of $S_{inv}$ is approximately inversely proportional to the distance between the coupling strength c and $1/4$, as shown in Fig. \ref{Fig.escapetime}.
The closer the coupling strength c is to $1/4$, the longer the average escape time.

\begin{theorem}\label{Th2}
	Consider the coupled map lattice
	\begin{eqnarray*}
		T: \left\{
		\begin{array}{ll}
			x_1(n+1) = (1-c)f(x_1(n)) + c f(x_2(n)), \\
			x_2(n+1) = c f(x_1(n)) + (1-c) f(x_2(n)),
		\end{array}
		\right.
	\end{eqnarray*}
	with the map in the following different forms (which are shown in Fig. \ref{3xmod1}):
	\begin{enumerate}
		\item $$	f(x) = \left\{
		\begin{array}{ll}
			3x, & 0 \leq x <1/3 , \\
			3x-1, & 1/3 \leq x < 2/3, \\
			3x - 2, & 2/3 \leq x \leq 1.
		\end{array}
		\right.  $$
		
		\item $$	f(x) = \left\{
		\begin{array}{ll}
			-3x+1, & 0 \leq x \leq  1/3, \\
			-3x+2, &  1/3 < x \leq  2/3, \\
			-3x +3, & 2/3 < x \leq 1.
		\end{array}
		\right.  $$
		
		
		
	\end{enumerate}
	
	Then $c=\frac{1}{3}$ ($c=\frac{2}{3}$ symmetrically) is the synchronization--intermittent synchronization transition point, that is,
	for all $c \in [0, \frac{1}{3}) \cup (\frac{2}{3}, 1]$, there exists a constant $r_0 > 0$ such that for almost every initial point $(x_1(0), x_2(0)) \in [0, 1]^2$, the following hold:
	\begin{eqnarray*}
		\left\{
		\begin{array}{ll}
			\liminf_{n \to \infty} \, dist(\bm{x}(n), S_{\text{inv}}) = 0 , \\
			\limsup_{n \to \infty} \, dist(\bm{x}(n), S_{\text{inv}}) \geq r_0 > 0 ;
		\end{array}
		\right.
	\end{eqnarray*}
	and for  $c\in (\frac{1}{3},\frac{2}{3})$,
	\[
	\lim_{n \to \infty} \, dist(\bm{x}(n), S_{\text{inv}}) = 0
	\]
	holds	for  every initial point.
	
	Moreover, when $c \in [0, \frac{1}{3}) \cup (\frac{2}{3}, 1]$, there exist a unique absolutely continuous invariant measure for $T$.
\end{theorem}

\begin{figure}
	\centering
	\includegraphics[width=0.5\textwidth]{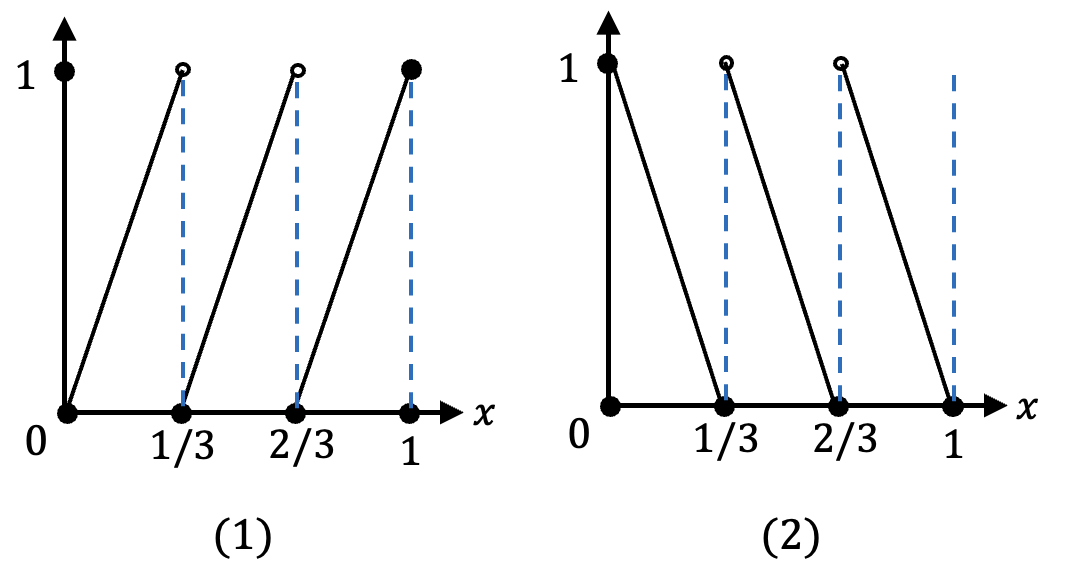}
	\caption{The figures show the map $f(x)=\pm 3x ~\text{mod}~ 1$ respectively. (1) corresponds to $f(x)= 3x \mod 1$,  while (2) corresponds to $f(x)= -3x ~\text{mod}~ 1$. }
	\label{3xmod1}
\end{figure}
We can also obtain the Lyapunov exponents of the diagonal invariant set $S_{inv}$:
\[
\lambda_{\parallel} = \ln 3, \quad \lambda_{\perp} = \ln |3(1-2c) |.
\]
When \( c \in (0, \frac{1}{3}) \),  \( \lambda_{\perp} = \ln |3(1-2c) | > 0 \), perturbations in the transverse direction grow exponentially, indicating that even if the initial point is arbitrarily close to the synchronized state, small random perturbations will cause it to rapidly diverge from the synchronization manifold.
When $c\in (\frac{1}{3},\frac{2}{3})$,  \( \lambda_{\perp} = \ln |3(1-2c) | <0 \), any  point through iterations will gradually get infinitely closer to $S_{inv}$ and never move away from it, indicating that synchronization occurs for $c\in (\frac{1}{3},\frac{2}{3})$.
Therefore,	we can conclude that $c = \frac{1}{3}$ ($c = \frac{2}{3}$ symmetrically) is the transition point between synchronization and intermittent-synchronization.
\vskip 0.2cm
\noindent {\bf Remark.}\ {\it The uniqueness of ACIM together with
Tsujii's result \cite{tsujii2001absolutely} (see the end of Section \ref{sec:4}) shows that the trajectory of almost each point in the sense of Lebesgue measure will visit almost everywhere of a region with a positive Lebesgue measure for infinite times.}

\section{\label{sec:2}Basic lemma}

In this section, we give  a fundamental iterative lemma of coupled map lattices.

We begin with some basic properties and definitions.
Roughly speaking, for any set $S$ with a small measure in some sense, due to the local expansion property of the map $T$,  we hope that the measure of $T^j(S)$ will become large enough for some sufficiently large $j$ such that $T^j \cap S_{inv} \neq \emptyset$.
If $T^i(S)$ is a segment or convex region for $i = 0, 1, \dots, j$, then we can show that for any neighborhood of the diagonal $S_{inv}$, there exists a constant $c_0 > 0$ such that there exists a subset $S_0$ of $S$ satisfying the following conditions:
\begin{itemize}
	\item For each point $\bm{q} \in S_0$, $T^j(\bm{q})$ lies in the neighborhood of $S_{inv}$;
	\item $M(S_0) \geq c_0 M(S)$, where $M(\cdot)$ denotes the Lebesgue measure or the length of a simple curve.
\end{itemize}
These two conditions imply that equations (\ref{equation:2}) and (\ref{equation:4}) hold for a set of full measure.

In reality, the situation is significantly more complex than the one outlined earlier, primarily because $T$ is not a one-to-one map and is not globally expanding. Specifically, since $T$ is not injective, it is generally the case that not all $T^i(S)$ for $i = 0, 1, \dots, j$ are segments or convex regions. Without this property, it becomes impossible to satisfy the second condition stated earlier.

To proceed, we decompose $[0, 1]^2$ as the union of small squares $Q_J$, where the squares $Q_J$ are defined by the partitions of $[0, 1]^2$ using the lines $x_1 = \frac{1}{2}$ and $x_2 = \frac{1}{2}$ when $f(x)= 2x ~\text{mod}~1$ as shown in Thm. \ref{Th1}, and using the lines $x_1 = \frac{1}{3}$, $x_1 = \frac{2}{3}$, $x_2 = \frac{1}{3}$ and $x_2 = \frac{2}{3}$ when $f(x)= \pm 3x ~\text{mod}~1$ as shown in Thm. \ref{Th2}.

For these piecewise linear expanding map $f$, we can observe that the map $T: Q_J \to [0, 1]^2$ is injective on each square $Q_J$.
Furthermore, $T(Q_J \cap S)$ is convex if $Q_J \cap S$ is.

Therefore, we divide the set $S$ into subsets of the form $S \cap Q_J$ and consider the iterations of $T$ on each subset individually.

\begin{definition}
	We define that a set $S$ has $i_0$ components if it can be decomposed into $i_0$ non-empty subsets of $S \cap Q_J$. Let $D \subset [0, 1]^2$ be a given set.
	For a set $S$ consisting of components $S_{0,1}, \dots, S_{0,i_0}$, we say that $S$ has $\hat{i}_0$ components within $D$ if there exist exactly $\hat{i}_0$ subsets, denoted by $S_{0,j}$, with $j = 1, \dots, \hat{i}_0 \leq i_0$, such that $S_{0,j} \cap D \neq \emptyset$ for each $1 \leq j \leq \hat{i}_0$.\label{key}
\end{definition}

\begin{definition}
	Given the set $S$ as described, suppose for each $1 \leq j \leq i_0$, the image of $S_{0,j}$ under $T$, denoted $T(S_{0,j})$, consists of $k_0(j)$ components. In this case, we say that the set $T(S)$ has a total of $\sum_{j=1}^{i_0} k_0(j)$ components.
	Similarly, suppose for each $1 \leq j \leq i_0$, $T(S_{1,j})$ has $\hat{k}_0(j)$ components in $D$. We then say that the set $T(S)$ contains $\sum_{j=1}^{\hat{i}_0} \hat{k}_0(j)$ components within $D$.
\end{definition}	

For Lorenz map $f(x)= 2x ~\text{mod}~1$ as an example, we consider two sets $S$ and $D$, as shown in Fig. \ref{Fig.diedai1}, $S$ has $2$ components as it has two non-empty subsets of $S\cap Q_{J}$.
We denote these subsets as $S_{0}$ and $S_{1}$.
$S$ has one component  in the set $D$ as $S_{0}\cap D \neq \emptyset $ and $S_{1}\cap D = \emptyset $.
In the same way, after one iteration, 	$T(S_{0})$ has $3$ components and one of them  intersects with  $D$;
$T(S_{1})$ has $3$ components and two of them  intersect with  $D$.
Therefore, $T(S)$ has $6$ components and three of them  intersect with $D$.
\begin{figure}[h]
	\centering
	\includegraphics[width=0.45\textwidth]{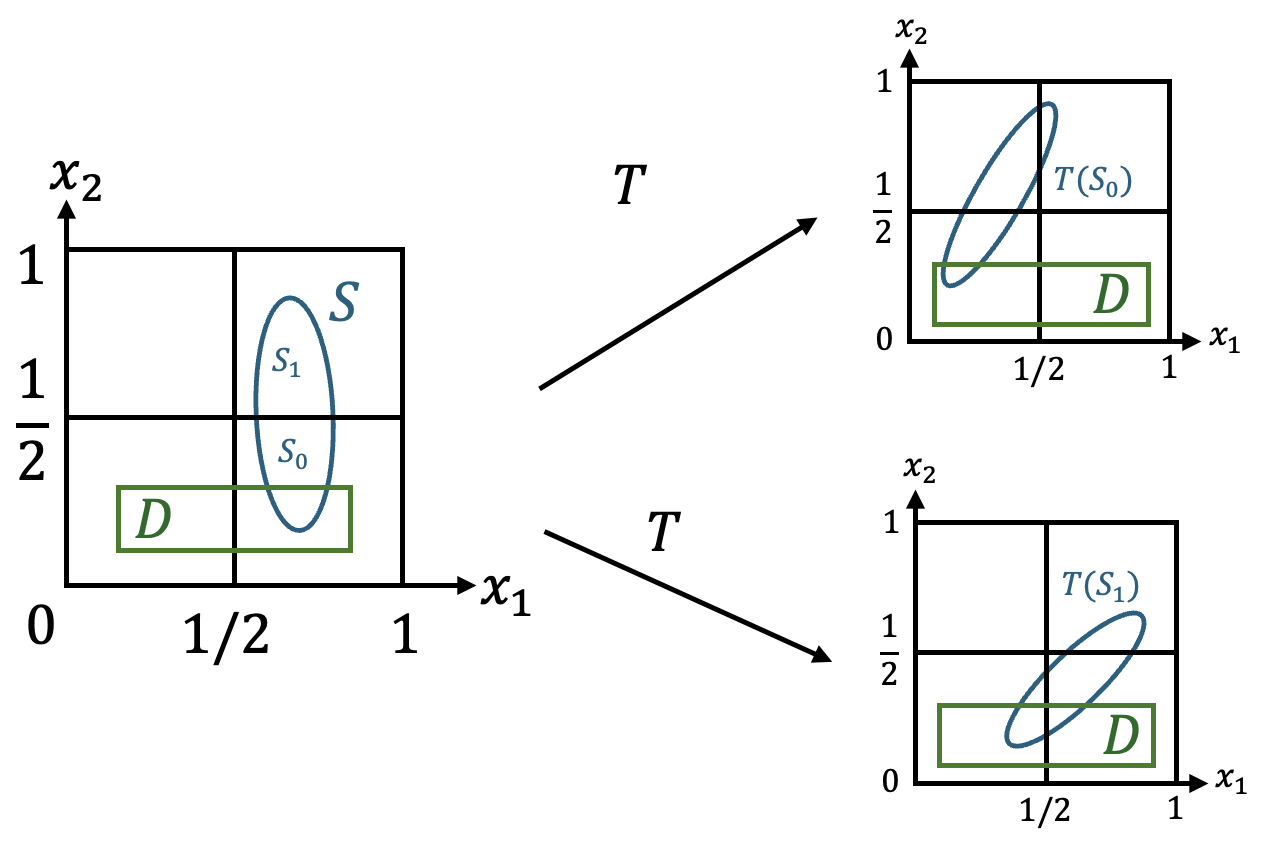}
	\caption{\label{Fig.diedai1} The figure shows an example of the number of components of $S$ and its iteration: $S$ has $2$ conponents that we define as  $S_{0},S_{1}$, one of them  intersects with $D$;
		$T(S_{0})$ has $3$ components and one of them  intersects with  $D$;
		$T(S_{1})$ has $3$ components and two of them  intersect with  $D$.
		Therefore, $T(S)$ has $6$ components and three of them intersect with $D$.}
\end{figure}

	Inductively, suppose $T^l(S)$ is composed of components $S_{l,1}, \dots, S_{l,i_l}$, and also contains components $S_{l,l_j}$, with $j = 1, \dots, \hat{i}_l \leq i_l$, inside $D$. Furthermore, for each $1 \leq j \leq i_l$, $T(S_{l,j})$ consists of $k_l(j)$ components, and for each $1 \leq j \leq \hat{i}_l$, $T(S_{l,l_j})$ contains $\hat{k}_l(j)$ components in $D$. We then say that $T^{l+1}(S)$ consists of $\sum_{j=1}^{i_l} k_l(j)$ components and $\sum_{j=1}^{\hat{i}_l} \hat{k}_l(j)$ components in $D$, respectively.

Moreover, we can apply the same approach to define the components of $T^{-l}(S)$ within $D$.
It is important to observe that for each component $\Omega$ of $T^l(S)$, there exists a subset $\Omega_0 \subset S$ such that $T^l: \Omega_0 \rightarrow \Omega$ is a homeomorphism.

Clearly, the measure of each component of $S$ is typically smaller than the measure of $S$ itself.
Additionally, the image of each component may be subdivided into multiple components. Therefore, it is necessary to establish that, on average, the local expansion of the map will eventually outweigh the division caused by the lines $x_1 = \frac{1}{2}$ and $x_2 = \frac{1}{2}$  for $f(x)= 2x ~\text{mod}~1$ and by the lines $x_1 = \frac{1}{3}$, $x_{1}=\frac{2}{3}$, $x_2 = \frac{1}{3}$ and $x_{2}=\frac{2}{3}$ for $f(x)= \pm 3x ~\text{mod}~1$ .
This ensures that the measures of the components continue to grow as the iteration progresses, but only if the intersection between the relevant set and $S_{inv}$ is empty.


Consider the coupled map lattice~$T$~in (\ref{model}),
let~$S \subset [0,1]^2$~be a simple curve or a  measurable set.
We define
\begin{eqnarray*}
	E_{+}(c) = \left\{
	\begin{array}{ll}
	\left| \det(I + cA) \right| k_+^2,  &\text{w.r.t. a measurable set } S, \\
		\underset{\bm{a} \in [0,1]^2}{\sup} \left\| JT(\bm{a}) \right\|,  &\text{w.r.t. a simple curve } S.
	\end{array}
	\right.
\end{eqnarray*}
where~$JT(\bm{a})$~is the Jacobian matrix of~$T$ at point~$\bm{a}$ and~$k_+ =\sup  \left| f'(x) \right|$.

Similarly, we define
\begin{eqnarray*}
	E_{-}(c) = \left\{
	\begin{array}{ll}
	\left| \det(I + cA) \right| k_-^2, &\text{w.r.t. a measurable set } S, \\
		\underset{\bm{a} \in [0,1]^2}{\inf} \left\| JT(\bm{a}) \right\|, & \text{w.r.t. a simple curve } S.
	\end{array}
	\right.
\end{eqnarray*}
where~$k_- = \inf  \left| f'(x) \right|$.

Obviously, for any simple curve or measurable set~$S$~in a small square~$Q_J$~of the phase space, we have
$$
E_{-}(c) M(S) \leq M(T(S)) \leq E_{+}(c) M(S).
$$
When $f(x)= 2x ~\text{mod}~1$ as shown in Thm. \ref{Th1}, we have~$E_{\pm}(c)=4(1-2c)$~for a measurable set and~$E_{+}(c) = 2$,~$E_{-}(c) = 2(1-2c)$~for a simple curve;
when $f(x)= \pm 3x ~\text{mod}~1$ as shown in Thm. \ref{Th2}, we have~$E_{\pm}(c)=9(1-2c)$~for a measurable set and~$E_{+}(c) = 3$,~$E_{-}(c) = 3(1-2c)$~for a simple curve.

	\begin{figure}[h]
	\centering
	\includegraphics[width=0.45\textwidth]{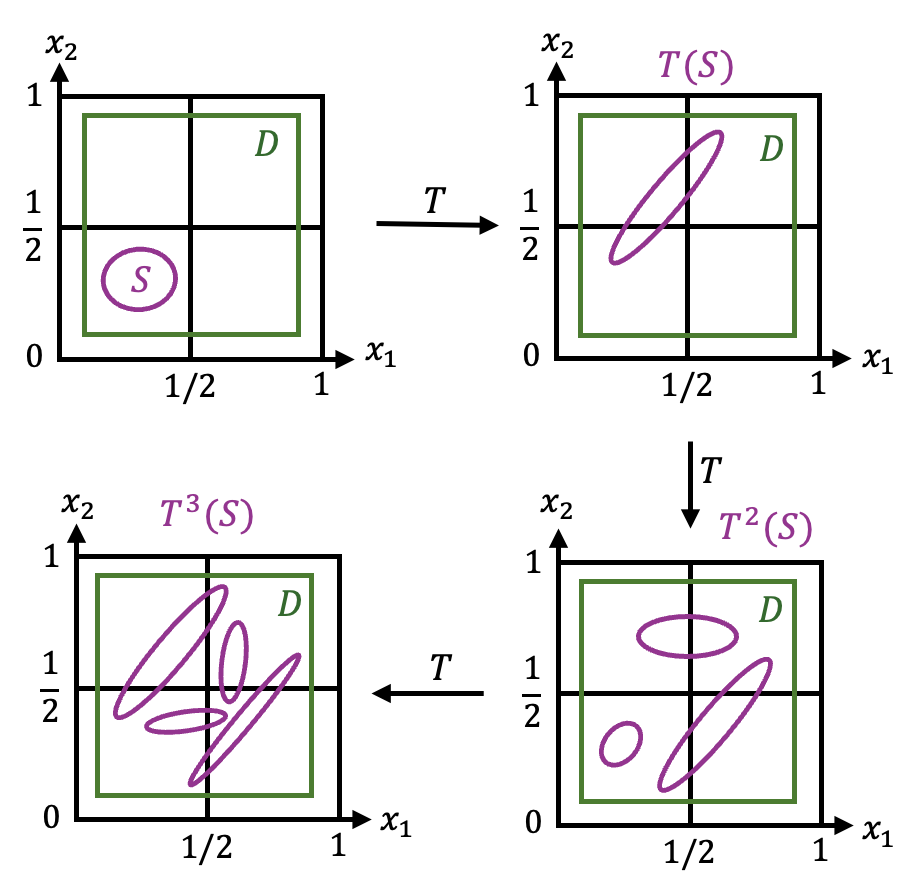}
	\caption{\label{Fig.diedai2} The figure shows an example of a measurable set $S$ in $D$ is $(\frac{1}{4},3,10)-$good as  $M(S)\leq \frac{1}{4} M(D)$, $T(S)$ has 3 components, $T^{2}(S)$ has 6 components and $T^{3}(S)$ has 10 components.}
\end{figure}

\begin{definition}
	Let~$\delta \in (0,1)$,~$m_0, a \in \mathbb{N}$,~and~$D$~be a domain in the phase space~$[0,1]^2$.
	We say that a simple curve or measurable set~$S$~in~$D$~is~$(\delta, m_0, a)$-good if~$S$~lies in some small square of~$[0,1]^2$~with~$M(S) \leq \delta M(D)$, and for each~$i \geq 0$~and each component~$S_1$~of~$T^j(S)$~with~$M(S_1) \leq \delta M(D)$, it holds that~$T^{m_0}(S_1)$~has at most~$a$~components.

\end{definition}
Note that for any~$S$~in some square of~$[0,1]^2$, as long as~$T^{m_0}(S)$~has at most~$a$~components, the second property of~$(\delta, m_0, a)$-good typically holds.
	For example, as shown in Fig. \ref{Fig.diedai2}, we can see $M(S)\leq \frac{1}{4} M(D)$ and $T^{3}(S)$ has 10 components, then we can say that the measurable set $S$ in $D$ is $(\frac{1}{4},3,10)-$good.


The following iteration lemma and its corollary follow from a proof mechanism similar to that of  Li et al.\cite{li2021intermittent}.
For the convenience of the readers and to maintain the integrity of the article, we present the complete proof in the Appendix.

\begin{lemma}\label{Lem:diedai}
	(Iterative Lemma)
	Assume~$E_{+}(c)\geq E_{-}(c)>1$.
	Let~$D$~be a domain in the phase space~$[0,1]^{2}$.
	Suppose~$0<\delta_{1}<1$~and~$a< E^{m_{0}}_{-}(c)$~with~$a,m_{0}\in \mathbb{N}$.
	Let
	\begin{eqnarray*}
	1<&\mu&<(1-\log_{E_{+}(c)}\frac{E_{-}(c)}{a^{1/m_{0}}})^{-1},\\
	d&=&1-(1-\log_{E_{+}(c)}\frac{E_{-}(c)}{a^{1/m_{0}}})\mu >0,\\
	F(c)&=&a\left(\frac{E_{-}(c)}{a^{1/m_{0}}}\right)^{1-\log_{E_{+}(c)}\delta_{1}},
	\end{eqnarray*}
	and define~$N_{0}=\lfloor \log_{\mu}(d^{-1}\log_{2}F(c))\rfloor$.
	
	Suppose~$N > N_{0}$~and for any $(\delta_{1},m_{0},a)$-good curve or measurable set~$\Omega \subset D$~with the range of measure~$M(\Omega)\in [2^{-\mu^{N+1}}M(D), 2^{-\mu^{N}}M(D)]$~and define~$k(N)=\lfloor -\log_{E_{+}(c)}\frac{M(\Omega)}{\delta_{1}M(D)}\rfloor$,
	then there exist some disjoint sub-curves or measurable subsets~$\Omega_{j}\subset \Omega$,~$j=1,\cdots,$~such that
	\begin{enumerate}
		\item for each~$j \geq 1$~it holds that~$T^{k(N)}(\Omega_{j})$~is a component of~$T^{k(N)}(\Omega)$~with~$M(T^{k(N)}(\Omega_{j}))\geq 2^{-\mu^{N}}M(D)$~and
		\item ~$M(\sum_{j\geq1}\Omega_{j})\geq (1-F(c)2^{-d\mu^{N}})M(\Omega)$.
	\end{enumerate}
\end{lemma}


\begin{corollary}\label{Coro:diedai}
	Let the domain~$D$,~$m_{0}$,~$a$,~$\mu$,~$N_{0}$~and~$d$~be as in  Lemma \ref{Lem:diedai}.
	Then there exists a constant~$c_{1}>0$~such that for any $(\delta_{1},m_{0},a)$-good curve or measurable set~$\Omega \subset D$~with~$M(\Omega)\leq \delta_{1}M(D)$, there exist some disjoint sub-curves or measurable subsets~$\Omega_{j}\subset \Omega$,~$j=1,\cdots,$~such that
	\begin{enumerate}
		\item for any~$i$, there exists~$k(i)$~such that~$T^{k(i)}(\Omega_{i})$~is a component of~$T^{k(i)}(\Omega)$~with~$M(T^{k(i)}(\Omega_{i}))\geq 2^{-\mu^{N_{0}}}M(D)$~and
		\item ~$M\left(\sum_{i}\Omega_{i}\right)\geq c_{1}M(\Omega)$.
	\end{enumerate}
\end{corollary}

Note that Corollary \ref{Coro:diedai} holds true for any  $1<\mu<(1-\log_{E_{+}(c)}\frac{E_{-}(c)}{a^{1/m_{0}}})^{-1}$.
	As~$\mu \rightarrow 1$, we obtain the upper bound~$\log_{E_{+}(c)}(E_{-}(c)a^{-1/m_{0}})F(c)^{-1}$~for~$2^{-\mu^{N_{0}}}$.

Note that  the Iterative Lemma \ref{Lem:diedai} is not limited to two-node CMLs: as long as the fundamental condition \( a < c<E^{m_{0}}_{-}(c)\) is satisfied, we can apply it to investigate intermittent-synchronization in multi-node CMLs.

\section{\label{sec:3}Intermittent-synchronization and unique ACIM of CML with the map $f(x)=2x \mod 1$}
In this section,  we consider two-node coupled standard Lorenz map lattice and obtain a complete range of coupling strength for intermittent-synchronization and the uniqueness of ACIM, i.e., Theorem \ref{Th1}.
We use a new geometric-combinatorics method to prove it.

\subsection{Proof of the intermittent-synchronization when  \( c \in [0, \frac{1}{4}) \).}

	 The proof of the intermittent-synchronization for Thm. \ref{Th1} can be divided into two parts: the ordered part (\ref{equation:4}) and the disordered part (\ref{equation:5});
	 while the process diagram for the proof of intermittent-synchronization is shown in Fig.  \ref{Fig.process}.
	\begin{figure}[h]
		\centering
		\includegraphics[width=0.45\textwidth]{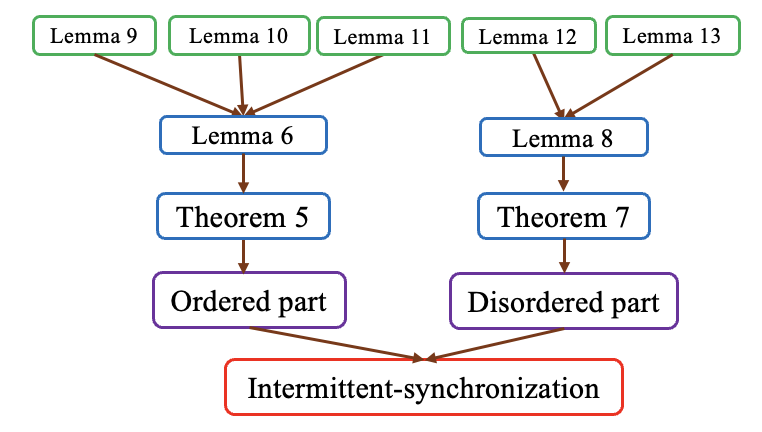}
		\caption{\label{Fig.process} The figure shows the process diagiam for the proof of intermittent-synchronization in Theorem \ref{Th1}, in which
		Lemma \ref{lemma6} is a corollary of Lemma 9, 10 and 11,
		Lemma \ref{lemma8} is a corollary of Lemma 12 and 13;
		Lemma 6 and 8 can derive Theorem \ref{thm5} and Theorem \ref{thm7} respectively;
		Theorem \ref{thm5} can derive the ordered part of intermittent-synchronization, and Theorem \ref{thm7} can derive the disordered part of intermittent-synchronization.}
	\end{figure}
	
	We define a neighborhood of the invariant set $S_{inv}$ as
	 $$
	 O_{\epsilon} = \{\bm{a} = (x_1, x_2) \in [0, 1]^2 \mid \text{dist}(\bm{a}, S_{inv}) \leq \epsilon \}
	 $$
	 and its complement as
	 $$
	 D_{\epsilon} = [0, 1]^2 \setminus O_{\epsilon},
	 $$
	 where~$S_{inv}$~is the invariant set and we give a position division of them in the phase space $[0,1]^{2}$,  as shown in Fig. \ref{Fig.neighborhood}.
	
	 \begin{figure}[h]
	 	\centering
	 	\includegraphics[width=0.48\textwidth]{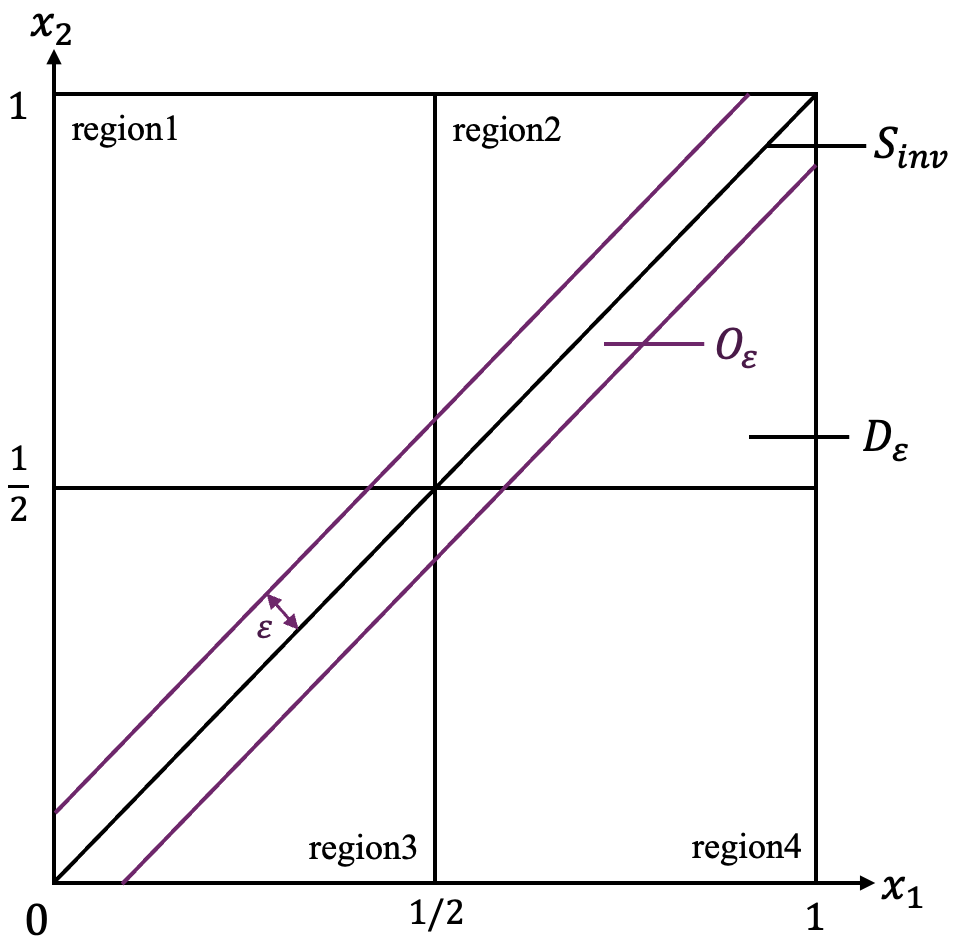}
	 	\caption{\label{Fig.neighborhood}The figure shows the set~$O_{\epsilon}$, its complement~$D_{\epsilon}$, and the invariant set~$S_{inv}$ in the phase space $[0,1]^{2}$.
	 	The four regions of phase space  $[0,1]^{2}$ cut by $x_{1}=\frac{1}{2}$ and $x_{2}=\frac{1}{2}$ are region 1, region 2, region 3 and region 4 respectively.}
	 \end{figure}
	
	 We first prove the ordered part of  Thm. \ref{Th1}, i.e., we will prove that equation (\ref{equation:4}) holds true for almost every initial point.
	 For this purpose, it is sufficient to prove the following theorem.
	
	 \begin{theorem}\label{thm5}
	 	Given any~$\epsilon > 0$, for almost every initial point~$(x_1(0), x_2(0))$~in~$[0, 1]^2$, there exists an~$n \in \mathbb{N}$~such that
	 	$$
	 	\text{dist}((x_1(n), x_2(n)), S_{\text{inv}}) < \epsilon.
	 	$$
	 \end{theorem}
	
Note that this result implies that, for almost every point, its orbit will eventually enter the~$\epsilon$-neighborhood of the invariant set $S_{inv}$.
	 	We will use this fact to show that, for almost every point, its orbit will repeatedly enter (or remain within) the~$\epsilon$-neighborhood of the invariant set $S_{inv}$ infinitely often, which is equivalent to the assertion in equation (\ref{equation:4}).

	 ~\\
	 \noindent \textbf{Proof  (\ref{equation:4}) from Theorem \ref{thm5}:}
	
	 For any positive integer~$i$, define the set
	 \begin{eqnarray*}
	 D_{i} &= &  \{(x_1^{(i)}(0), x_2^{(i)}(0)) \in [0,1]^2 \mid \text{dist}((x_1^{(i)}(0), x_2^{(i)}(0)), S_{inv})\\
	 &&\geq \frac{1}{i} \text{ for all } n \in \mathbb{N}\}~.
	 \end{eqnarray*}
	 To establish equation (\ref{equation:4}), it suffices to demonstrate that $M\left( \cup_{i \geq 1} D_i \right) = 0$.
	 Since $S_{inv}$ is an invariant set, we define $\epsilon = \frac{1}{i}$. This leads to $M(D_i) = 0$, which consequently implies that
	 $M\left( \cup_{i \geq 1} D_i \right) = 0$.
	 \qed

	 	Theorem \ref{thm5} is a corollary of the following lemma.
	
	 \begin{lemma}\label{lemma6}
	 	For any \( 0 \leq c < \frac{1}{4} \) and \( \epsilon > 0 \), there exists \( 0<c_0 \equiv c_0(c, \epsilon) <1 \) with the following property: for any convex set \( S \subset [0, 1]^2 \), one can find a countable collection of disjoint measurable subsets \( \{S_i\}_{i=1}^\infty \) of \( S \) satisfying:
	 	\begin{enumerate}
	 		\item For each \( i \in \mathbb{N} \), there exists an iteration time \( l_i \in \mathbb{N} \) such that \( T^{l_i}(S_i) \subset O_{\epsilon} \);
	 		\item The measure of the union satisfies \( M\left(\bigcup_{i=1}^\infty S_i\right) \geq c_0 M(S) \).
	 	\end{enumerate}
	 \end{lemma}
	
	  \noindent \textbf{Proof  Theorem \ref{thm5} from Lemma \ref{lemma6}:}
	
	 	Let \( S \subset [0,1]^2 \) be an arbitrary convex set.
	 	By Lemma \ref{lemma6}, there exists a collection of disjoint measurable subsets \( \{S_i^{(1)}\}_{i=1}^\infty \) satisfying properties $\mathit{1}$ and $\mathit{2}$.
	 	The remainder set \( S \setminus \bigcup_{i=1}^\infty S_i^{(1)} \) can be decomposed into countably many disjoint convex subsets \( \{S_j'\}_{j=1}^\infty \), with the measure estimate:
	 	\[ M\left(\bigcup_{j=1}^\infty S_j'\right) \leq (1 - c_0)M(S). \]
	 	
	 	We proceed inductively.
	 	At the \( n \)-th stage, applying Lemma \ref{lemma6} to each \( S_j' \) yields subsets \( \{S_{j,k}^{(n)}\} \) such that:
	 	\[ M\left(\bigcup_{k=1}^\infty S_{j,k}^{(n)}\right) \geq c_0 M(S_j') \]
	 	and \( T^{l(j,k)}(S_{j,k}^{(n)}) \subset O_\epsilon \) for some \( l(j,k) \in \mathbb{N} \).
	 	The remaining portion satisfies:
	 	\[ M\left(S \setminus \bigcup_{k=1}^\infty S_{j,k}^{(n)}\right) \leq (1 - c_0)^n M(S). \]
	 	
	 	Taking the limit as \( n \to \infty \), we conclude that the set of points in \( S \) whose orbits never enter \( O_\epsilon \) has measure zero. An application of Fubini's Theorem completes the proof of Theorem \ref{thm5}.
      \qed
	
	 To prove equation (\ref{equation:5}), which corresponds to the disordered part of Thm. \ref{Th1}, we need to show that for almost every point in the phase space, its orbit, while visiting any neighborhood of $S_{inv}$, will also remain arbitrarily distant from $S_{inv}$ infinitely often. It suffices to demonstrate that

	 \begin{theorem}\label{thm7}
	 	For almost every point~$(x_{1}(0),x_{2}(0))$~in the phase space, there exists an~$n(x_{1}(0),x_{2}(0))$~such that~$(x_{1}(n),x_{2}(n))\in D_{r_{0}}$~with~a constant $r_{0}>0$.
	 \end{theorem}
	
	 \noindent \textbf{Proof (\ref{equation:5}) from Theorem \ref{thm7}:}
	
	 Let
	 $$
	 S_0 = \left\{ (x_1(0), x_2(0)) \in [0,1]^2 \mid (x_1(i), x_2(i)) \in O_{r_0} \text{ for all } i \right\}.
	 $$
	 Then, by  Thm. \ref{thm7} we have that \( M(S_0) = 0 \).
	
	 Let
	 $$
	 S_n = \left\{ (x_1(0), x_2(0)) \in [0,1]^2 \mid (x_1(i), x_2(i)) \in O_{r_0} \text{ for all } i > n \right\}
	 $$
	 be the subset of \( [0,1]^2 \) such that for each point \( \bm{q} \in S_n \) and each \( i > n \), \( T^i(\bm{q}) \) always stays in \( O_{r_0} \).
	
	 From the definition, we have that \( T^{n+1}(S_n) \subset S_0 \).
	
	 If \( M(S_n) > 0 \), then there exists \( \widetilde{S}_n \subset S_n \) with \( M(\widetilde{S}_n) > 0 \) such that the map \( T^{n+1}: \widetilde{S}_n \to T^{n+1}(\widetilde{S}_n) \subset S_0 \) is a diffeomorphism.
	
	 This implies
	 $$
	 M(S_0) \geq M(T^{n+1}(\widetilde{S}_n)) > 0.
	 $$
	 However, this contradicts the fact that \( M(S_0) = 0 \). Hence, \( M(S_n) = 0 \) for each \( n \), which leads to the  equation  (\ref{equation:5}).
	 \qed
	
	 The proof of Theorem~\ref{thm7} reduces to establishing the following statement by Fubini's Theorem:
	 for  almost every segment $L \subset [0,1]^2$ with slope $-1$, the orbit of almost every point $x \in L$ satisfies $T^n(x) \in D_{r_0}$ for some finite iteration time $n \in \mathbb{N}$.
	 Therefore, it is sufficient to prove the following lemma.
	 \begin{lemma}\label{lemma8}
	 	For any $0 \leq c < \frac{1}{4}$, there exists $0 < c_1 \equiv c_1(c) < 1$ such that for almost each segment $L$ with a slope $-1$ in $[0, 1]^2$, there exist its disjoint subsegments $L_1, L_2, \cdots$, satisfying
	 	\begin{enumerate}
	 		\item for any $i$, there exists $l_i \geq 0$ such that $T^{l_i}(L_i) \subset D_{r_0}$;
	 		\item  $M(\cup_i L_i) \geq c_1 M(L)$.
	 	\end{enumerate}
	 \end{lemma}

	 Lemma \ref{lemma6} and Lemma \ref{lemma8} are corollaries of the following Lemmas.
	
\begin{lemma}\label{lemma9}
	If ~$\Omega$~ is a convex set in one of the four small regions of~$[0,1]^{2}$~and~$T(\Omega)$~has an intersection with each small region simultaneously, then~$(\frac{1}{2},\frac{1}{2})$~lies in~$T(\Omega)$.
	
	Moreover, we have that for any~$\epsilon>0$, ~$\frac{M(T(\Omega)\cap O_{\epsilon})}{M(T(\Omega))}\geq \frac{1}{2}\epsilon^{2}$~holds true, where~$O_{\epsilon}=\{\bm{p}\in [0,1]^{2} | dist( \bm{p}, S_{inv})\leq \epsilon\}$.
\end{lemma}

\begin{proof}
	We have set four regions of~$[0,1]^{2}$~as region 1, 2, 3 and 4, as shown in Fig. \ref{Fig.neighborhood}.
	Because~$T(\Omega)$~intersects with four regions, the points~$A_{J_{1}}$,~$A_{J_{2}}$,~$A_{J_{3}}$~and~$A_{J_{4}}$~in each of the four regions belong to~$T(\Omega)$.
	From the nature of the convex set, we can know that the convex hull determined by these points is a subset of~$T(\Omega)$~and then the point~$(\frac{1}{2},\frac{1}{2})$~is in it.
	
	Moreover, we set~$H_{\epsilon}$~be the line~$$\{\bm{p}\in[0,1]^{2}|dist(\bm{p},S_{inv})=\epsilon\}$$~and define~$\overline{H}_{\epsilon}=H_{\epsilon}\cap T(\Omega)$.
	Furthermore, we set~$L$~be the point set of the union of all lines connecting~$\overline{H}_{\epsilon}$~and the point~$(\frac{1}{2},\frac{1}{2})$.
	From the above conclusion, it can be seen that ~$(\frac{1}{2},\frac{1}{2})\in T(\Omega)$~and if~$\overline{H}_{\epsilon}$~is empty, then~$T(\Omega)\subset O_{\epsilon}$~and the proof is complete from the convexity of~$T(\Omega)$.
	Thus, we can assume that~$\overline{H}_{\epsilon}$~is nonempty. then we get~$T(\Omega)\backslash(L\cap T(\Omega))\subset T(\Omega)\cap O_{\epsilon}$~obviously.
	
	Then we need to focus the analysis on~$L \cap T(\Omega)$.
	Again by the convexity, we have that~$L \cap O_{\epsilon}\subset T(\Omega)$~which further implies~$L \cap O_{\epsilon}=(L\cap T(\Omega))\cap O_{\epsilon}$.
  For our purpose, it is sufficient to estimate~$\frac{M(L \cap O_{\epsilon})}{M(L \cap T(\Omega))}$. Note that~$L \cap T(\Omega)\subset L\cap [0,1]^{2}$. For each line~$l\in L$, It is easily to see that the length~$l\cap [0,1]^{2}$~is less than~$\sqrt{2}$~while the length of~$l\cap O_{\epsilon}$~is larger than~$\epsilon$. thus we have that:
$$
		M((L\cap T(\Omega))\cap O_{\epsilon} )\\  \geq (\frac{\epsilon}{\sqrt{2}})^{2}M(L \cap [0,1]^{2})
		\geq \frac{\epsilon^{2}}{2}M(L\cap T(\Omega))
$$
	Since ~$\frac{M( T(\Omega)\cap O_{\epsilon} )}{M(T(\Omega))}>\frac{M((L\cap T(\Omega))\cap O_{\epsilon} )}{M(L\cap T(\Omega))}\geq \frac{\epsilon^{2}}{2}$,  we have completed the whole certificate.
\end{proof}	
	
\begin{lemma}\label{lemma10}
		For any $0 \leq c < \frac{1}{4}$, $O_{\epsilon}$ (the $\epsilon$-neighborhood of $S_{inv}$ defined above) and any convex set~$\Omega_{0}$~in one of the four regions with measure less than~$0.5\cdot 10^{-4} \epsilon^{2}$, there exist constants ~$\theta_{1}, \alpha_{1}>0$~and disjoint convexities~$\Omega_{i}\subset \Omega_{0} ,i=1,2,\cdots$~satisfying the following two conclusions:
		\begin{enumerate}
			\item for any~$i\geq 1$, there exists a~$t(i)$~such that~$T^{t(i)}(\Omega_{i})$~is a component of~$T^{t(i)}(\Omega_{0})$~with~$M(T^{t(i)}(\Omega_{i}))\geq \theta_{1}$;
			\item ~$M(\underset{i\geq 0}{\cup}\Omega_{i})\geq \alpha_{1}M(\Omega_{0})$.
		\end{enumerate}
	\end{lemma}
	
	\begin{proof}
	We note that if a convex set~$\Omega_{0}$~of~$[0,1]^{2}$~intersects with  the segments~$x_{1}=x_{2}$,  these two conditions obviously hold true.
	At the same time, $x_{2}=x_{1}\pm \frac{1}{2}$ are the pre-images of $x_{1}=x_{2}$,
	therefore, we can assume that~$T(\Omega_{0})$~and its iterations do not intersect with either~$x_{1}=x_{2}$~or~$x_{2}=x_{1}\pm \frac{1}{2}$, as shown in Fig. \ref{frontimage1}.
		\begin{figure}
		\centering
		\includegraphics[width=0.4\textwidth]{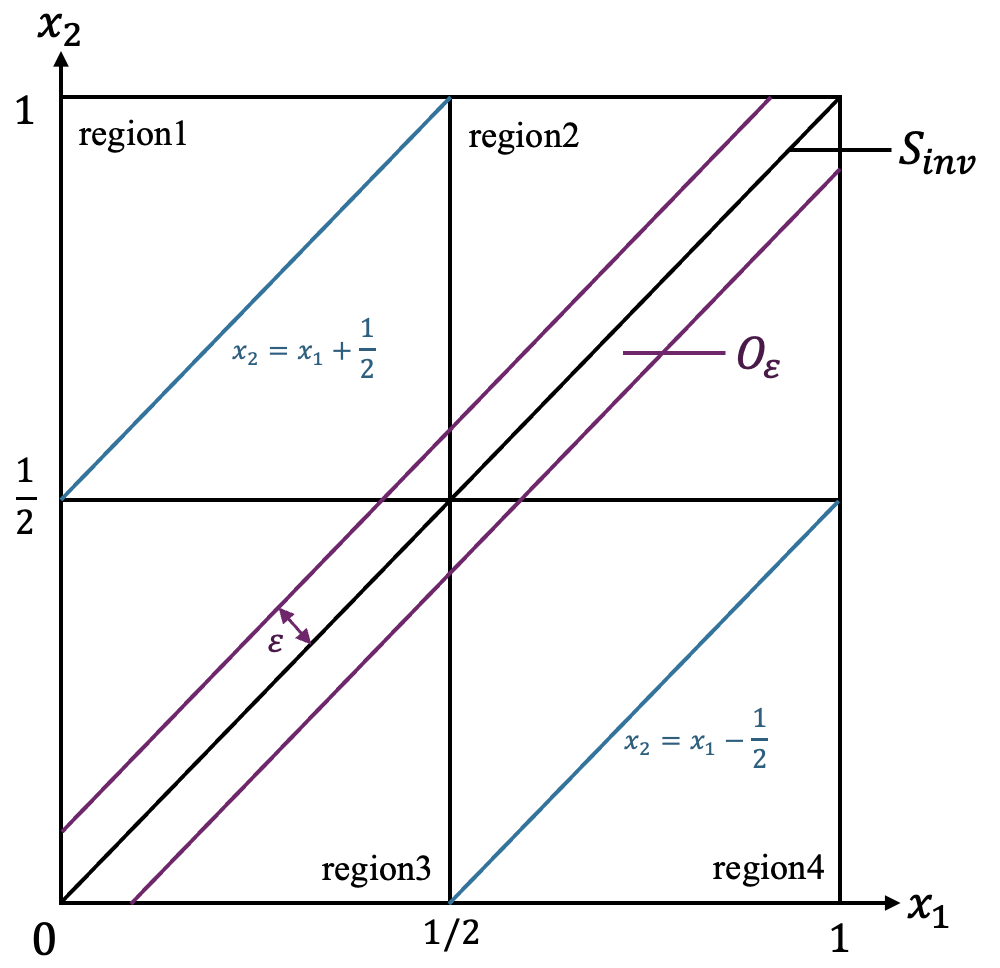}
		\caption{The figure shows the pre-image of $x_{1}=x_{2}$:  $x_{2}=x_{1}\pm \frac{1}{2}$.
We can assume that~$T(\Omega_{0})$~and its iterations do not intersect with either~$x_{1}=x_{2}$~or~$x_{2}=x_{1}\pm \frac{1}{2}$.	}
		\label{frontimage1}
	\end{figure}
	
	For any convex set~$\Omega_{0}$,
	we define a rectangle~$S_{\Omega_{0}}$ such that it contains $\Omega_{0}$ and its  long side is parallel to the longest line segment between boundary points of $\Omega_{0}$ (denoted as~$l_{long}$), with both having the same length, and its short sides intersect with the two endpoints of~$l_{long}$, as shown in Fig. \ref{rectangle}.
	We claim that $0.5M(S_{\Omega_{0}}) \leq M(\Omega_{0}) \leq M(S_{\Omega_{0}})$.
	In fact, we set the two intersection points of the long side of $S_{\Omega_{0}}$ and the convex set~$\Omega_{0}$ to be $A_{1}$ and $A_{2}$ respectively, and set the perpendicular segment from point $A_{1}$ to segment $l_{long}$ as $h_{1}$,  the perpendicular segment from point $A_{2}$ to segment $l_{long}$ as $h_{2}$, as shown in Fig. \ref{rectangle}.
	In addition, we set the triangles with $l_{long}$ as the base, $h_{1}$ and $h_{2}$ as the height as $\triangle \overline{A}_{1}$ and $\triangle \overline{A}_{2}$ respectively.
	Due to the convexity of $\Omega_{0}$, the measure of $\Omega_{0}$ is smaller than the sum of the measures of $\triangle \overline{A}_{1}$ and $\triangle \overline{A}_{2}$.
	Note that the sum of the measures of $\triangle \overline{A}_{1}$ and $\triangle \overline{A}_{2 }$ is $0.5M(S_{\Omega_{0}})$, then the claim holds true.
	\begin{figure}
		\centering
		\includegraphics[width=0.4\textwidth]{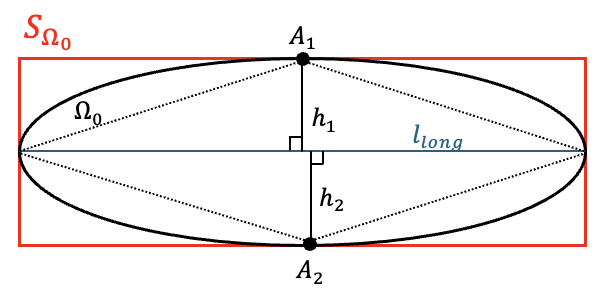}
		\caption{The figure shows a convex set~$\Omega_{0}$~and the corresponding rectangle~$S_{\Omega_{0}}$. The long side of the rectangle~$S_{\Omega_{0}}$~to be parallel to the longest line segment between boundary points of the convex set (denoted as~$l_{long}$), with both having the same length. }
		\label{rectangle}
	\end{figure}
	
	In order to calculate the number of components of the convex set $\Omega_{0}$ after iterations, we only need to calculate the number of components of the corresponding rectangle $S_{\Omega_{0}}$ after iterarions.
    We know that~$M(\Omega_{0}) < 0.5\cdot 10^{-4} \epsilon^{2}$, hence~$M(S_{\Omega_{0}}) \leq 2M(\Omega_{0}) <10^{-4}  \epsilon^{2}$.
    We set the length of~$S_{\Omega_{0}}$~to~$d_{l}$~and the width to~$d_{w}$.
	
	\textbf{Case i}:  If $d_{w} <d_{l}<\epsilon$, the length of $T(S_{\Omega_{0}})$ is less than $2\epsilon$,
	this indicate that $T(S_{\Omega_{0}})$ has at most 2 components, otherwise $T(S_{\Omega_{0}})\cap O_{\epsilon}\neq \emptyset$, which means that the conclusions are valid, as shown in Fig. \ref{case1}.
	Therefore, in this case, $T(\Omega_{0})$ has at most 2 components and $M(T(\Omega_{0}))>4(1-2c)M(\Omega_{0})>2 M(\Omega_{0})$,  we can apply Iteration Lemma \ref{Lem:diedai} and Coro. \ref{Coro:diedai} with $m_{0}=1$, $a=2$, thus we complete the proof of this lemma.
	
		\begin{figure}
		\centering
		\includegraphics[width=0.4\textwidth]{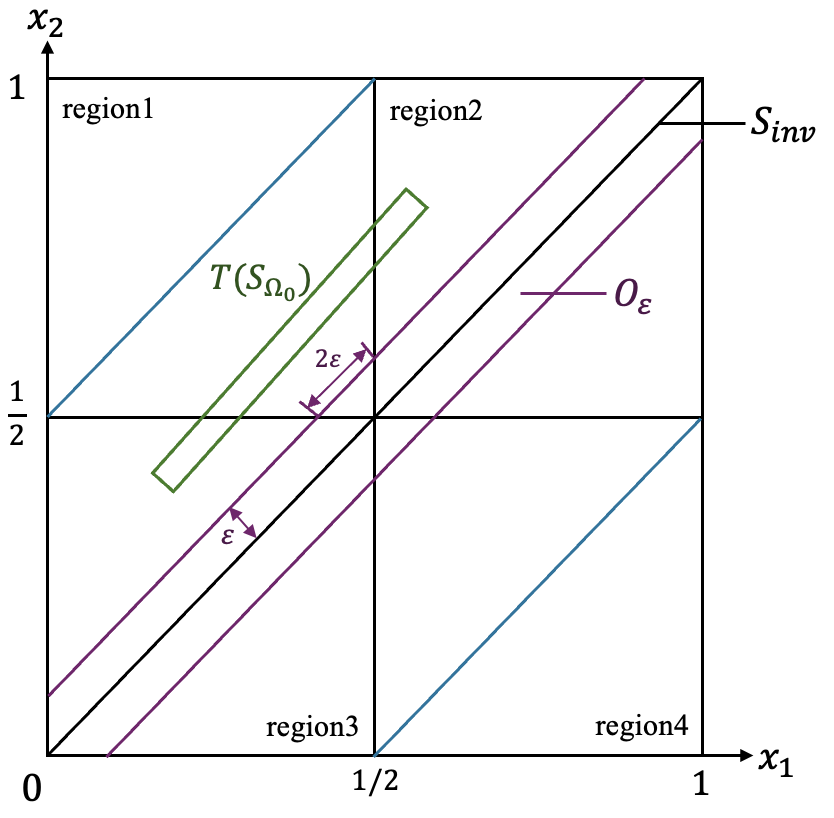}
		\caption{The figure shows that when $d_{w}<d_{l}<\epsilon$, if $T(S_{\Omega_{0}})$ has 3 components, then its length must be larger than $2\epsilon$.
		Therefore $T(S_{\Omega_{0}})$  has at most 2 components.	}
		\label{case1}
	\end{figure}
	
\textbf{Case ii}: If $d_{l} > \epsilon$, then it follows that $d_{w} < 10^{-4} \epsilon$.
For any line in $[0,1]^{2}$ with the slope $k$, regardless of the region where the line is located, after one iteration, the slope changes to:
$$
k \stackrel{T}{\longrightarrow} \frac{c + (1-c)k}{(1-c) + ck}.
$$
Additionally, when $k = 0$ or $k = \infty$, the slope after iterations will not approach $0$ or $\infty$.
Thus, for any line in $[0,1]^{2}$, the slopes of all components of the line after each iteration  are the same.
Furthermore, if the components of the line after some iteration are parallel to $x_{1} = 1/2$ or $x_{2} = 1/2$, they will not approach  parallel to $x_{1} = 1/2$ or $x_{2} = 1/2$ again after iterations.

Let the long edge of $S_{\Omega_{0}}$ be $D_{L}$, then we have $|D_{L}|=d_{l}>\epsilon$.
For every component in $T^{i}(S_{\Omega_{0}})$, we set the edge parallel to $T^{i}(D_{L})$ as  'good edge'.
Then we set
$$
d^{i}_{l}=\sum |\text{'good~ edge' ~in ~every ~component} \in T^{i}(S_{\Omega_{0}}) |.
$$
 Clearly,  under normal circumstances, we have
$$
2(1-2c) d^{i-1}_{l} \leq d^{i}_{l} \leq  2 d^{i-1}_{l}.
$$
However, if $T^{i}(S_{\Omega_{0}})\cap \{x_{1}=\frac{1}{2}\}$ (or $\{x_{2} = \frac{1}{2}\}$) $\neq \emptyset$ and
the 'good edges' of the components of $T^{i}(S_{\Omega_{0}})$  are parallel to  $x_{1} = \frac{1}{2}$ or $x_{2} = \frac{1}{2}$, then we have $d^{i}_{l} < 4 d^{i-1}_{l}$.

Now we assume that after $j$ iterations of $S_{\Omega_{0}}$, the 'good edges' of all components are parallel to either $x_{1} = \frac{1}{2}$ or $x_{2} = \frac{1}{2}$, then $d^{j}_{l} < 4 d^{j-1}_{l}$.
Therefore, for $i > j$, we have
\begin{eqnarray*}
	d^{i}_{l} &\leq &  2^{i-j}  d^{j}_{l} \\
	&\leq &  2^{i-j}  \cdot  4 d^{j-1}_{l} \\
	&\leq&  2^{i-j}  \cdot  4 \cdot 2^{j-1} d_{l} \\
	&=&  2^{i+1} d_{l}.
\end{eqnarray*}

For $T^{i}(S_{\Omega_{0}})$, we define
\begin{eqnarray*}
M_{i}=& \#&\{ \text{ 'long~component' }\in   T^{i}(S_{\Omega_{0}})\mid \\
 &| &\text{ the 'good edge' of the component} ~~|\geq \epsilon \},
\end{eqnarray*}
\begin{eqnarray*}
	m_{i}=& \#&\{ \text{ 'short~component' }\in   T^{i}(S_{\Omega_{0}})\mid \\
	&| & \text{the 'good edge' of the component}~~| < \epsilon \}.
\end{eqnarray*}
Thus, we have $M_{i} \cdot \epsilon \leq d^{i}_{l} \leq 2^{i+1} d_{l}$, which implies that $M_{i} \leq \frac{d_{l}}{\epsilon} \cdot 2^{i+1}$.
In addition, any 'long component' can produce up to two 'short components' after one iteration, and any 'short component' can also produce up to two 'short components' after one iteration.
Therefore, we have $m_{i} \leq 2 m_{i-1} + 2 M_{i-1}$.

Since $d_{l} > \epsilon$, we have $M_{0} = 1$ and $m_{0} = 0$.
For any $i$, this leads to
\begin{eqnarray*}
	m_{i} &\leq &  2 m_{i-1} + 2 M_{i-1} \\
	&\leq &  2^{2} m_{i-2} + 2^{2} M_{i-2} + 2 M_{i-1} \\
	&\leq&  \sum^{i-1}_{k=0} 2^{i-k} M_{k} \\
	&\leq & \frac{d_{l}}{\epsilon}  \sum^{i-1}_{k=0} 2^{i+1} = \frac{d_{l}}{\epsilon} \cdot i \cdot 2^{i+1}.
\end{eqnarray*}
Therefore, the total number of components of $T^{i}(\Omega_{0})$  is less than
$$
M_{i} + m_{i} \leq \frac{d_{l}}{\epsilon} \cdot (i+1) \cdot 2^{i+1} = \frac{2d_{l}}{\epsilon} \cdot (i+1) \cdot 2^{i}.
$$
For the fixed $c\in (0,\frac{1}{4})$, we have $E_{-}(c) = 4(1-2c)>2$.
Let $N_{0} (c)= \min \{ i \in \mathbb{N} \mid (4(1-2c))^{i} > \frac{2d_{l}}{\epsilon} \cdot (i+1) \cdot 2^{i} \}$, then we can apply Iteration Lemma \ref{Lem:diedai} and Corollary \ref{Coro:diedai} with $m_{0} = N_{0}(c)$, $a = \frac{2d_{l}}{\epsilon} \cdot (N_{0}(c)+1) \cdot 2^{N_{0}(c)}$, and $E_{-}(c) = 4(1-2c)$. Thus, we complete the proof of this lemma.

	 \end{proof}
	\begin{lemma}\label{lemma11}
		For any $0\leq c < \frac{1}{4}$, $O_{\epsilon}$ (the $\epsilon$-neighborhood of $S_{inv}$ defined above),~$\theta_{1}>0$ defined above, and any convex set~$\Omega_{0}$~in one of the four regions with measure~$M(\Omega_{0})\geq \theta_{1}$, there exists a constant~$\alpha_{2}>0$~depending on~$\epsilon$~and disjoint convexities~$\Omega_{i}\subset \Omega_{0} ,i=1,2,\cdots$~satisfying the following two conclusions:
		\begin{enumerate}
			\item for any~$i\geq 1$, there exists a~$t(i)$~such that~$T^{t(i)}(\Omega_{i})\subset O_{\epsilon}$;
			\item ~$M(\underset{i\geq 0}{\cup}\Omega_{i})\geq \alpha_{2}M(\Omega_{0})$.
		\end{enumerate}
	\end{lemma}
	
\begin{proof}
As discussed in Lemma \ref{lemma10}, suppose any convex set and its iterative convex set do not intersect with ~$x_{1}=x_{2}$~and~$x_{2}=x_{1}\pm \frac{1}{2}$, as shown in Fig. \ref{frontimage1}.

If $\Omega_{0}$ is divided into at most two components after each iteration, there exists a component $\Omega_{1} \in T(\Omega_{0})$ such that
$$
M(\Omega_{1}) > \frac{1}{2} \cdot 4(1-2c) \cdot M(\Omega_{0}) = 2(1-2c) M(\Omega_{0}).
$$
Thus, we can construct a sequence of convex sets defined by $\Omega_{j+1} = T(\Omega_{j}) \cap ~\text{region}~i$ for some $i$ such that
$$
M(\Omega_{j+1}) > 2(1-2c) M(\Omega_{j}) > (2(1-2c))^{j} M(\Omega_{0}).
$$
Since \( M(\Omega_{0}) \geq  \theta_{1} \), it follows that \( (2(1-2c))^{t} M(\Omega_{0}) > 1 \) for some fixed \( t = t(c) \).
Thus, there exists some \( j < t \) such that \( T(\Omega_{j}) \cap \text{region }i \neq \emptyset \) for every $i=1,2,3,4$.
Consequently, by applying Lemma \ref{lemma9}, we can conclude that the results hold true.

If $\Omega_{0}$ is first divided into three components after $a$ iterations, then $T^{a}(\Omega_{0})$ intersects with $x_{1} + x_{2} = 1$.
Since \( M(\Omega_{0}) \geq \theta_{1} \), we have $|T^{a}(\Omega_{0}) \cap \{ x_{1} + x_{2} = 1 \}| \geq  \theta_{a} > 0$.
Let $\Omega_{a}$ be the component in $T^{a}(\Omega_{0})$ that intersects with $x_{1} + x_{2} = 1$, then we have $|\Omega_{a} \cap \{ x_{1} + x_{2} = 1 \}| \geq  \theta_{a} > 0$.
Since  $x_{1} + x_{2} = 1$ is an invariant set for $T$, $\Omega_{a}$ continues to intersect with $x_{1} + x_{2} = 1$ after each iteration.
Therefore, there exists a component $\Omega^{1}_{a} \subset  T(\Omega_{a})$ such that
\begin{eqnarray*}
|\Omega^{1}_{a} \cap \{ x_{1} + x_{2} =1 \}| &\geq & 2(1-2c)\cdot  |\Omega_{a} \cap \{ x_{1} + x_{2} = 1 \}| \\
&\geq & 2(1-2c) \theta_{a}.
\end{eqnarray*}
Thus,  we can construct a sequence of convex sets defined by $\Omega^{j+1}_{a}\subset  T(\Omega^{j}_{a})$  such that
\begin{eqnarray*}
|\Omega^{j+1}_{a} \cap \{ x_{1} + x_{2} = 1 \}| &\geq & 2(1-2c)\cdot |\Omega^{j}_{a} \cap \{ x_{1} + x_{2} = 1 \}| \\
&\geq & (2(1-2c))^{j} \theta_{a}.
\end{eqnarray*}
It follows that \( (2(1-2c))^{t} \theta_{a} > \frac{\sqrt{2}}{4} \) for some fixed \( t = t(c) \).
Thus, there exists some \( j < t \) such that \( (\frac{1}{2}, \frac{1}{2}) \in \Omega^{j}_{a} \).
Consequently, by applying Lemma \ref{lemma9} we can conclude that the results hold true.
\end{proof}


\begin{lemma}\label{lemma12}
		For any $0\leq c < \frac{1}{4}$, there exists a constant \( \theta_{2} \in (0,1) \) such that for almost every segment \( l \) in one of the four  regions with slope \(  - 1 \) and measure \( M(l) < \theta_{2} \), there exist constants \( \alpha_{3}, \theta_{3} > 0 \) and \( r_{0} \) depending on \( \theta_{2} \), as well as disjoint subsegments \( l_{i} \), \( i=1,2,\ldots \), such that the following conclusions hold:
	\begin{enumerate}
		\item For any \( i \geq 1 \), there exists a \( t(i) \geq 0 \) such that \( T^{t(i)}(l_{i}) \) is a component of \( T^{t(i)}(l) \) with \( M(T^{t(i)}(l_{i})) \geq \theta_{3} \) or \( T^{j}(l_{i}) \subset D_{r_{0}} \) for some \( j \leq t(i) \);
		\item \( M\left(\bigcup_{i} l_{i}\right) \geq \alpha_{3} M(l) \).
	\end{enumerate}
\end{lemma}

\begin{proof}
	Let \( O_{r} = \{\bm{x} \in [0,1]^{2} \mid \text{dist}(\bm{x}, S_{inv}) \leq r\} \) for \( r > 0 \), and denote \( \overline{O}_{r} = \{\bm{x} \in O_{r} \mid \text{dist}(\bm{x}, \{x_{1} + x_{2} = 1\}) \leq r\} \).
	Consider the segment \( l   \subset O_{r_{0}} \)  with length less than \( \theta_{2} \).
	We fix \( c \) such that \( 2|1-2c| > 1 \), let \( M = \lfloor \log_{2(1-2c)}3\rfloor + 1 \), and choose \( \theta_{2} \) sufficiently small and \( r_{0} \) satisfying
	 \( (2(1-2c))^{i} \theta_{2} <2r_{0} < 2^{-i} \) for \( i =1, 2, 3, \ldots, M \).
	 For a segment \( l \subset O_{r_{0}} \), we claim that \( T^{M}(l) \) has at most three components in $O_{r_{0}} $.
	
	Assuming \( T(l) \) has three components leads to \( T(l) \cap \overline{O}_{r_{0}} \neq \emptyset \). Since \( \left(\frac{1}{2}, \frac{1}{2}\right) \in \overline{O}_{r_{0}} \) and the diameter of \( \overline{O}_{r_{0}} \) is less than \( 2r_{0} \), we have  that \( T(l) \) is in the \(r_{0} \)-neighborhood of \( \left(\frac{1}{2}, \frac{1}{2}\right) \).
	Thus, we find that \( T^{i}(l_{i}) \cap O_{r_{0}} \) is in the \((2^{i-1}r_{0} )\)-neighborhood of $(0,0)$ or $(1,1)$ for \( 2 \leq i \leq M \).
	Noting that $2^{i-1}r_{0}<\frac{1}{4}$, we conclude that the \((2^{i-1}r_{0} )\)-neighborhood of \( (0,0) \) (or $(1,1)$) and \( \overline{O}_{r_{0}} \) do not intersect, thus proving our claim.
	
	Since \( (2(1-3c))^{M} > 3 \), we apply Iteration Lemma \ref{Lem:diedai} and Coro. \ref{Coro:diedai} with \( m_{0} = M \), \( a =3\), and \( E_{-}(c) = 2(1-2c) \) to complete the proof of this lemma.
\end{proof}

	 \begin{lemma}\label{lemma13}
	 	For any~$0 \leq c < \frac{1}{4}$,~$\theta_3 >0$ stated above, and~$r_0 = \frac{\theta_3}{3}$, consider a segment~$l_0$ in one of the four regions with slope~$-1$ and $M(l_{0})\geq \theta_3 $.
	 	Then there exists a   subsegments~$l^{'}_0$ of ~$l_0$  such that  $l^{'}_0\subset D_{r_0}$ and $M(l^{'}_0) \geq \frac{1}{3} M(l_0)$.
	 \end{lemma}

	 \begin{proof}
	 If $l_{0}\cap O_{r_{0}} = \emptyset$, the conclusion naturally follows.
	
	 If $l_{0}\cap O_{r_{0}} \neq \emptyset$,
	 since the segment~$l_0$ has slope~$-1$ and~$M(l_0)\geq  \theta_3 =3r_0$, as shown in Fig. \ref{Fig.disorder-segment}, the maximum length of the intersection of $l_{0}$ and $O_{r_{0}}$ is $2r_{0}$, then in this case the length of $l_{0}$ outside $O_{r_{0}}$ is larger than
	 $$
	 \frac{M(l_{0})-2r_{0}}{M(l_{0})}=1- \frac{2r_{0}}{M(l_{0})}\geq 1-\frac{2r_{0}}{3r_{0}}=\frac{1}{3}.
	 $$
	that is $M(l^{'}_0) \geq \frac{1}{3} M(l_0)$.
	 	\begin{figure}
	 	\centering
	 	\includegraphics[width=0.4\textwidth]{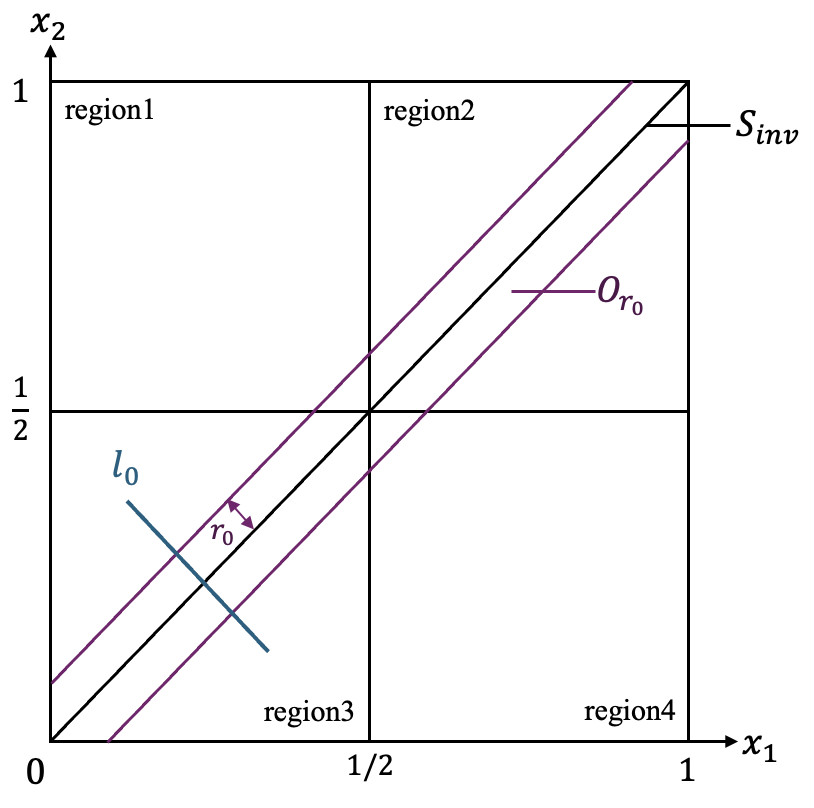}
	 	\caption{ This figure shows that if $l_{0}\cap O_{r_{0}} \neq \emptyset$, then $|l_{0}\cap  O_{r_{0}}| \leq 2r_{0}$.}
	 	\label{Fig.disorder-segment}
	 \end{figure}
	
	 \end{proof}
	 From Lemma \ref{lemma9}, \ref{lemma10} and \ref{lemma11}, we can directly derive Lemma \ref{lemma6}, thereby establishing the ordered part of intermittent-synchronization.
	 Similarly, Lemma \ref{lemma8} follows from Lemma \ref{lemma12} and \ref{lemma13}, leading to the derivation of the disordered part of intermittent-synchronization.
	 \subsection{Proof of the uniqueness of absolutely continuous invariant measure when  \( c \in [0, \frac{1}{4}) \).}
	
	 \begin{lemma}\label{lemma14}
	 	Let the origin point  \( O = (0, 0) \),  the point \( M = \left( \frac{1}{4}, \frac{1}{4} \right) \) and $\epsilon>0$ depending on $T$.
	 	When \( c \in [0, \frac{1}{4}) \), for any open set \( G \), there exists an integer \( n \) such that \( T^n G \) contains a $\epsilon$-neighborhood of \( M \).
	 \end{lemma}

	 \begin{proof}
	 	By Lemma \ref{lemma10} and Lemma \ref{lemma11}, we know that for any open set \( G \),  there exists a component that intersects the diagonal  \( x_1 = x_2 \) after iterations.
	 	We denote the convex set of this component after iterations as \( \widehat{G} \) such that \( \widehat{G} \) intersects with the diagonal.
	 	
	 	Before \( \widehat{G} \) intersects the point \( (1/2, 1/2) \) under iterations, the length of its intersection with the set \( \{x_1 = x_2\} \) increases by a factor of \( 2 \) with each iteration, it follows that \( \widehat{G} \) will eventually intersect \( (1/2, 1/2) \) under iterations.
	 	
	 	Let \( i_0 \in \mathbb{N} \) be the number of the iterations such that \( (1/2, 1/2) \notin T^{n}(\widehat{G}) \), $n=1,2, \cdots, i_{0}-1$ and \( (1/2, 1/2) \in T^{i_0}(\widehat{G}) \), then we can define an open set \( G_0 \subset T^{i_0}(\widehat{G}) \) such that  \( (1/2, 1/2) \in G_0 \).
	 	Let \( l_0 = |G_0 \cap \{x_1 = x_2\}| > 0 \), and we define \( G_1 = G_0 \cap \text{region 2} \) and  \(l_{1}= |G_1 \cap \{x_1 = x_2\}|>0  \).
	 	Then $O=(0,0)\in TG_{1}$ and $ |TG_1 \cap \{x_1 = x_2\}|=2l_{1} $, as shown in Fig. \ref{unique1}.
	 	
	 	If \( M \in TG_1 \), then the conclusion holds directly.
	 	Therefore, we only need to consider  \( M \notin TG_1 \).
	 	In this case, we have
	 	\[
	 	|T^{2}G_1 \cap \{x_1 = x_2\}|=2^{2}l_{0} \quad \text{and} \quad O=(0,0)\in T^{2}G_{1}
	 	\]
	 \begin{figure}
	 	\centering
	 	\includegraphics[width=0.5\textwidth]{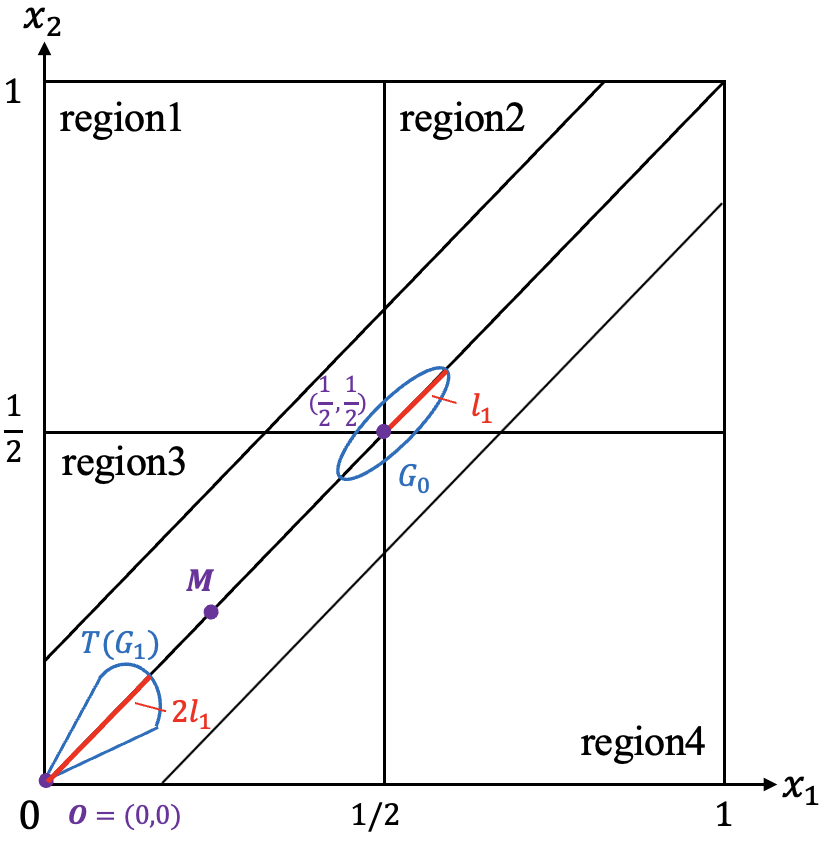}
	 	\caption{The figure shows that $(1/2, 1/2) \in  G_0 $, $G_1 = G_0 \cap~ \text{region} ~2$ and \( l_{1}=|G_1 \cap \{x_1 = x_2\}| \).}
	 	\label{unique1}
	 \end{figure}
	 Since $O$ is a fixed point for $T$, then for any $i\in \mathbb{N}$, we have
	 	\[
	 |T^{i}G_1 \cap \{x_1 = x_2\}|=2^{i}l_{0} \quad \text{and} \quad O=(0,0)\in T^{i}G_{1}
	 \]
	 Therefore, there exist $j\in \mathbb{N}$ such that
	 	\[
	 |T^{j}G_1 \cap \{x_1 = x_2\}|=2^{j}l_{0} >|OM|
	 \]
	 	which implies \( M \in T^{j}G_1 \), and the conclusion follows.
	 \end{proof}
	
	 Note that for $0\le c<1/4$, each eigenvalue of $T$ possesses an absolute value strictly larger than 1 at every smooth point.
	 From Theorem 3 in Tsujii \cite{tsujii2001absolutely}, we conclude that for CML with standard Lorenz map $f(x)=2x ~ \text{mod}~ 1$, there exist finitely many absolutely continuous ergodic probability measures \( \mu_i \), where \( 1 \leq i \leq p \), and the basin of each measure \( \mu_i \), defined as
	 \[
	 \text{Basin}(\mu_i) = \left\{ x \in [0,1]^2 \mid \frac{1}{n} \sum_{j=0}^{n-1} \delta_{T^j(x)} \xrightarrow{w} \mu_i \right\},
	 \]
	 is an open set modulo sets of Lebesgue measure zero, and the union \( \bigcup_{i=1}^p \text{Basin}(\mu_i) \) has full Lebesgue measure in \( [0,1]^2 \).
	
	 For two distinct absolutely continuous ergodic measures \( \mu_1 \) and \( \mu_2 \), we define two small balls \( B_1 \subset \text{Basin}(\mu_1) \) and \( B_2 \subset \text{Basin}(\mu_2) \).
	 From Lemma \ref{lemma14}, we know that after iterations, both \( B_1 \) and \( B_2 \) will contain neighborhoods of \( M \), denoted as \( M_{\epsilon_1} \) and \( M_{\epsilon_2} \) respectively.
	 Therefore, both \( B_1 \) and \( B_2 \) will contain a common neighborhood of \( M \)  after iterations, specifically \( M_{\epsilon_1} \cap M_{\epsilon_2} \).
	 This shows that basin of different absolutely continuous ergodic measure $\mu_{i}$ will contain the same open set modulo a Lebesgue null set after iterations.
	 Therefore, there can be only one absolutely continuous ergodic probability measure,  thereby proving the uniqueness of the measure.

	
	 \section{\label{sec:4}Intermittent-synchronization and unique ACIM of CML with the map $f(x)=\pm 3x ~\text{mod}~ 1$}
	 In this section,  we consider two-node CML where the map is  $f(x)=\pm 3x ~\text{mod}~1$, i.e., Thm. \ref{Th2}.
	 Similar to Sec. \ref{sec:3}, we also obtain a complete range of coupling strength for intermittent-synchronization and prove the uniqueness of ACIM for CML.
	

	\subsection{Proof of the intermittent-synchronization when  \( c \in [0, \frac{1}{3}) \).}

	 To prove the intermittent-synchronization of Thm. \ref{Th2}, we also need some lemmas to prove Lemma \ref{lemma6} and  Lemma \ref{lemma8}  in Sec. \ref{sec:3}.
	
	\begin{lemma}\label{lemma15}
		If ~$\Omega$~ is a convex set in one of the nine regions~of~$[0,1]^{2}$~and~$T(\Omega)$~has an intersection with each region simultaneously, then~$(\frac{1}{2},\frac{1}{2})$~lies in~$T(\Omega)$.
		
		Moreover, we have that for any~$\epsilon>0$, ~$\frac{M(T(\Omega)\cap O_{\epsilon})}{M(T(\Omega))}\geq \frac{1}{2}\epsilon^{2}$~holds true, where~$O_{\epsilon}=\{\bm{p}\in [0,1]^{2} | dist( \bm{p}, S_{inv})\leq \epsilon\}$.
	\end{lemma}

	\begin{proof}
		We set  nine regions of phase space  $[0,1]^{2}$ cut by $x_{1}=1/3$,  $x_{1}=2/3$  and $x_{2}=1/3$, $x_{2}=2/3$ as region 1-9 respectively, as shown in Fig. \ref{3xmod1frontimage}.
		Because~$T(\Omega)$~intersects with these nine regions, the points~$A_{J_{1}},\cdots, A_{J_{9}}$ in each of the nine regions belong to~$T(\Omega)$.
		From the nature of the convex set, we can know that the convex hull determined by these points is a subset of~$T(\Omega)$~and then the point~$(\frac{1}{2},\frac{1}{2})$~is in it.
		
		Moreover, we set~$H_{\epsilon}$~be the line
		$$
		\{\bm{p}\in[0,1]^{2}|dist(\bm{p},S_{inv})=\epsilon\}
		$$
		and define~$\overline{H}_{\epsilon}=H_{\epsilon}\cap T(\Omega)$.
		Furthermore, we set~$L$~be the point set of the union of all lines connecting~$\overline{H}_{\epsilon}$~and the point~$(\frac{1}{2},\frac{1}{2})$.
		From the above conclusion, it can be seen that ~$(\frac{1}{2},\frac{1}{2})\in T(\Omega)$~and if~$\overline{H}_{\epsilon}$~is empty, then~$T(\Omega)\subset O_{\epsilon}$~and the proof is complete from the convexity of~$T(\Omega)$.
		Thus, we can assume that~$\overline{H}_{\epsilon}$~is nonempty. then we get~$T(\Omega)\backslash(L\cap T(\Omega))\subset T(\Omega)\cap O_{\epsilon}$~obviously.
		
		Then we need to focus the analysis on~$L \cap T(\Omega)$. Again by the convexity, we have that~$L \cap O_{\epsilon}\subset T(\Omega)$~which further implies~$L \cap O_{\epsilon}=(L\cap T(\Omega))\cap O_{\epsilon}$.
		In order to complete our certificate, it is sufficient to estimate~$\frac{M(L \cap O_{\epsilon})}{M(L \cap T(\Omega))}$. Note that~$L \cap T(\Omega)\subset L\cap [0,1]^{2}$. For each line~$l\in L$, It is easily to see that the length~$l\cap [0,1]^{2}$~is less than~$\sqrt{2}$~while the length of~$l\cap O_{\epsilon}$~is larger than~$\epsilon$. thus we have that:
		$$
		M((L\cap T(\Omega))\cap O_{\epsilon} )\\  \geq (\frac{\epsilon}{\sqrt{2}})^{2}M(L \cap [0,1]^{2})
		\geq \frac{\epsilon^{2}}{2}M(L\cap T(\Omega))
		$$
		Since ~$\frac{M( T(\Omega)\cap O_{\epsilon} )}{M(T(\Omega))}>\frac{M((L\cap T(\Omega))\cap O_{\epsilon} )}{M(L\cap T(\Omega))}\geq \frac{\epsilon^{2}}{2}$,  we have completed the whole certificate.
	\end{proof}	
	
	 \begin{lemma}\label{lemma16}
	 	For any $0\leq c <\frac{1}{3}$, $O_{\epsilon}$ (the $\epsilon$-neighborhood of $S_{inv}$ defined above) and any convex set~$\Omega_{0}$~in one of the nine regions with measure less than~$0.5\cdot 10^{-4} \epsilon^{2}$, there exist constants ~$\delta_{1}, \beta_{1}>0$~and disjoint convexities~$\Omega_{i}\subset \Omega_{0} ,i=1,2,\cdots$~satisfying the following two conclusions:
	 	\begin{enumerate}
	 		\item for any~$i\geq 1$, there exists a~$t(i)$~such that~$T^{t(i)}(\Omega_{i})$~is a component of~$T^{t(i)}(\Omega_{0})$~with~$M(T^{t(i)}(\Omega_{i}))\geq \delta_{1}$;
	 		\item ~$M(\underset{i\geq 0}{\cup}\Omega_{i})\geq \beta_{1}M(\Omega_{0})$.
	 	\end{enumerate}
	 \end{lemma}
	
	 \begin{proof}
	 	We note that if a convex set~$\Omega_{0}$~of~$[0,1]^{2}$~intersects with ~$x_{1}=x_{2}$ or the pre-images of $x_{1}=x_{2}$,  these two conditions obviously hold true.
	 	Therefore, we can assume that~$T(\Omega_{0})$~and its iterations do not intersect with either~$x_{1}=x_{2}$~or the pre-images of $x_{1}=x_{2}$, as shown in Fig. \ref{3xmod1frontimage}.
	 	
	 		\begin{figure}
	 		\centering
	 		\includegraphics[width=0.4\textwidth]{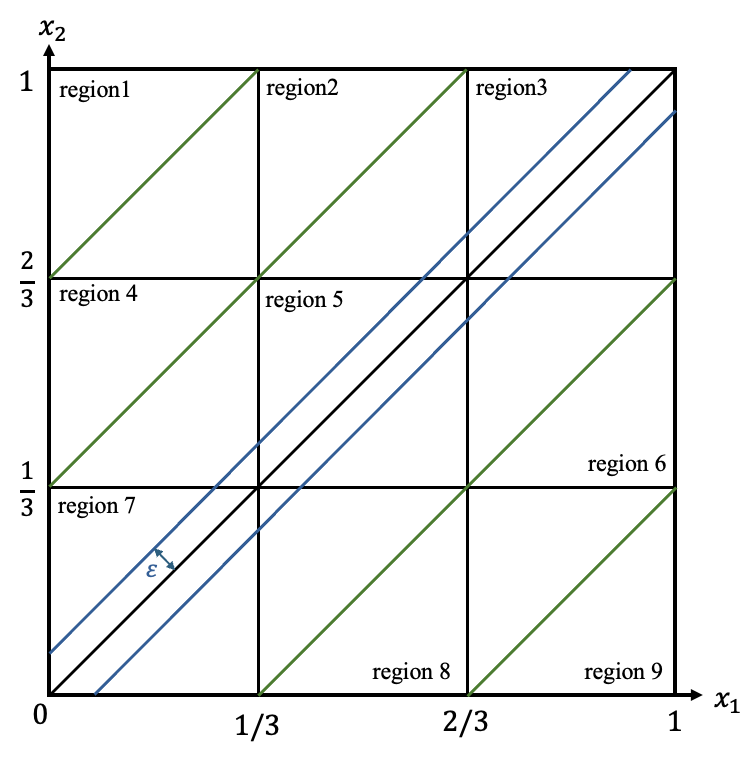}
	 		\caption{The figure shows the pre-images of $x_{1}=x_{2}$ which are shown as green lines.
	 			We can assume that~$T(\Omega_{0})$~and its iterations do not intersect with either~$x_{1}=x_{2}$~or~its pre-images.
	 			The nine regions of phase space  $[0,1]^{2}$ cut by $x_{1}=\frac{1}{3}$,  $x_{1}=\frac{2}{3}$  and $x_{2}=\frac{1}{3}$, $x_{2}=\frac{2}{3}$ are region 1-9 respectively.}
	 		\label{3xmod1frontimage}
	 	\end{figure}
	 	
	 	Similar to Lemma \ref{lemma10}, for any convex set~$\Omega_{0}$, we can also define the corresponding rectangle as $S_{\Omega_{0}}$ and $S_{\Omega_{0}}$~satisfies~$\frac{1}{2}M(S_{\Omega_{0}}) \leq M(\Omega_{0}) \leq M(S_{\Omega_{0}})$, as shown in Fig. \ref{rectangle}.
	 	
	 	In order to calculate the number of components of the convex set $\Omega_{0}$ after iterations, we only need to calculate the number of components of the corresponding rectangle $S_{\Omega_{0}}$ after iterarions.
	 	We know that~$M(\Omega_{0}) <0.5\cdot 10^{-4} \epsilon^{2}$, hence~$M(S_{\Omega_{0}}) \leq 2M(\Omega_{0}) < 10^{-4}\epsilon^{2}$.
	 	We set the length of~$S_{\Omega_{0}}$~to~$d_{l}$~and the width to~$d_{w}$.
	 	
	 		\begin{figure}
	 		\centering
	 		\includegraphics[width=0.48\textwidth]{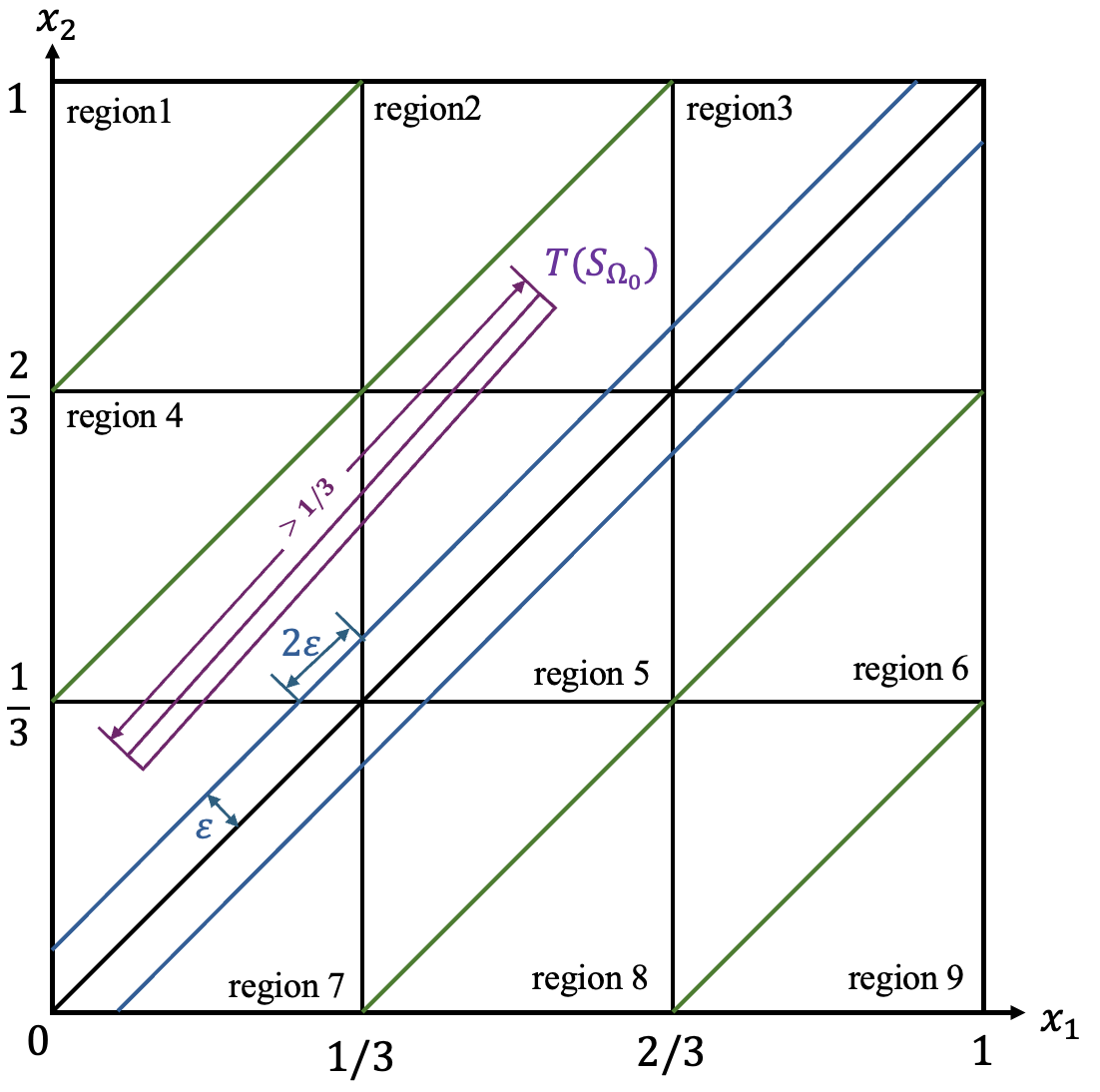}
	 		\caption{The figure shows that when $d_{w}<d_{l}<\epsilon$, if $T(S_{\Omega_{0}})$ has at least 4 components, then its length must be larger than $\frac{1}{3}$.
	 			Therefore $T(S_{\Omega_{0}})$  has at most 3 components.}
	 		\label{3xmod1case1}
	 		
	 	\end{figure}
	 	\textbf{Case i}:  If $d_{w}<d_{l}<\epsilon$, the length of $T(S_{\Omega_{0}})$ is less than $3\epsilon$,
	 	this indicate that $T(S_{\Omega_{0}})$ has at most 3 components, as shown in Fig. \ref{3xmod1case1}.
	 	Therefore, in this case, $T(\Omega_{0})$ has at most 3 components and $M(T(\Omega_{0}))>9(1-2c)M(\Omega_{0})>3 M(\Omega_{0})$,  we can apply Iteration Lemma \ref{Lem:diedai} and Corollary \ref{Coro:diedai} with $m_{0}=1$, $a=3$, thus we complete the proof of this lemma.
	 	
	 	\textbf{Case ii}: If $d_{l} > \epsilon$, then it follows that $d_{w} <  10^{-4}\epsilon$.
	 	For any line in $[0,1]^{2}$ with the slope $k$, regardless of the region where the line is located, after one iteration, the slope changes to
	 	$$
	 	k \stackrel{T}{\longrightarrow} \frac{c + (1-c)k}{(1-c) + ck}.
	 	$$
	 	Additionally, when $k = 0$ or $k = \infty$, the slope after iterations will not approach $0$ or $\infty$.
	 	Thus, for any line in $[0,1]^{2}$, the slopes of all components of the line after each iteration  are the same.
	 	Furthermore, if the components of the line after some iteration are parallel to $x_{1} = 1/3$ or $x_{2} = 1/3$, they will not approach parallel to  $x_{1} = 1/3$ or $x_{2} = 1/3$ again after iterations.
	 	
	 	Let the long edge of $S_{\Omega_{0}}$ be $D_{L}$, then we have $|D_{L}|=d_{l}>\epsilon$.
	 	For every component in $T^{i}(S_{\Omega_{0}})$, we set the edge parallel to $T^{i}(D_{L})$ as  'good edge'.
	 	Then we set
	 	$$
	 	d^{i}_{l}=\sum |\text{'good~ edge' ~in ~every ~component} \in T^{i}(S_{\Omega_{0}}) |.
	 	$$
	 	Clearly,  under normal circumstances, we have
	 	$$
	 	3(1-2c) d^{i-1}_{l} \leq d^{i}_{l} \leq  3 d^{i-1}_{l}.
	 	$$
	 	However, if $T^{i}(S_{\Omega_{0}})\cap \{x_{1}=\frac{1}{3}\}$ (or $\{x_{1} = \frac{2}{3}\}$, $\{x_{2} = \frac{1}{3}\}$, $\{x_{2} = \frac{2}{3}\}$) $\neq \emptyset$ and
	 	the 'good edges' of the components of $T^{i}(S_{\Omega_{0}})$  are parallel to $x_{1} = 1/3$ (or $x_{1} = 2/3$,  $x_{2} = 1/3$, $x_{2} = 2/3$), then we have $d^{i}_{l} < 6 d^{i-1}_{l}$.
	 	
	 	Now we assume that after $j$ iterations of $S_{\Omega_{0}}$, the 'good edges' of all components are parallel to $x_{1} = 1/3$ (or $x_{1} = 2/3$,  $x_{2} = 1/3$, $x_{2} = 2/3$).
	 	then $d^{j}_{l} < 6 d^{j-1}_{l}$.
	 	Therefore, for $i > j$, we have
	 	\begin{eqnarray*}
	 		d^{i}_{l} &\leq &  3^{i-j}  d^{j}_{l} \\
	 		&\leq &  3^{i-j}  \cdot  6 d^{j-1}_{l} \\
	 		&\leq&  3^{i-j}  \cdot  6 \cdot 3^{j-1} d_{l} \\
	 		&=&  2 \cdot 3^{i} d_{l}.
	 	\end{eqnarray*}
	 	
	 	For $T^{i}(S_{\Omega_{0}})$, we define
	 	\begin{eqnarray*}
	 		M_{i}=& \#&\{ \text{ 'long~component' }\in   T^{i}(S_{\Omega_{0}})\mid \\
	 		&|&\text{ the 'good edge' of the component} ~~|\geq \epsilon \},
	 	\end{eqnarray*}
	 	\begin{eqnarray*}
	 		m_{i}=& \#&\{ \text{ 'short~component' }\in   T^{i}(S_{\Omega_{0}}) \mid \\
	 		&| &\text{the 'good edge' of the component}~~| < \epsilon \}.
	 	\end{eqnarray*}
	 	Thus, we have $M_{i} \cdot \epsilon \leq d^{i}_{l} \leq 2 \cdot 3^{i} d_{l}$, which implies that $M_{i} \leq \frac{d_{l}}{\epsilon} \cdot 2 \cdot 3^{i}$.
	 	In addition, any 'long component' can produce up to two  'short components' after one iteration, and any 'short component' can also produce up to two  'short components' after one iteration.
	 	Therefore, we have $m_{i} \leq 2 m_{i-1} + 2 M_{i-1}$.
	 	
	 	Since $d_{l} > \epsilon$, we have $M_{0} = 1$ and $m_{0} = 0$.
	 	For any $i$, this leads to
	 	\begin{eqnarray*}
	 		m_{i} &\leq &  2 m_{i-1} + 2 M_{i-1} \\
	 		&\leq &  2^{2} m_{i-2} + 2^{2} M_{i-2} + 2 M_{i-1} \\
	 		&\leq&  \sum^{i-1}_{k=0} 2^{i-k} M_{k} \\
	 		&\leq & \frac{d_{l}}{\epsilon} 2^{i+1}  \sum^{i-1}_{k=0} (\frac{3}{2})^{k} = \frac{d_{l}}{\epsilon}\cdot 2^{i+1}\cdot 2 ((\frac{3}{2})^{i} -1)\\
	 		&<&\frac{4d_{l}}{\epsilon}\cdot 3^{i}.
	 	\end{eqnarray*}
	 	Therefore, the total number of components of $T^{i}(\Omega_{0})$  is less than
	 	$$
	 	M_{i} + m_{i} \leq \frac{d_{l}}{\epsilon} \cdot (2+4) \cdot 3^{i} = \frac{6d_{l}}{\epsilon}\cdot 3^{i}.
	 	$$
	 	For the fixed $c\in (0,\frac{1}{4})$, we have $E_{-}(c) = 9(1-2c)>3$.
	 	Let $N_{0} (c)= \min \{ i \in \mathbb{N} \mid (9(1-2c))^{i} > \frac{6d_{l}}{\epsilon} \cdot 3^{i} \}$, then we can apply Iteration Lemma \ref{Lem:diedai} and Corollary \ref{Coro:diedai} with $m_{0} = N_{0}(c)$, $a = \frac{6d_{l}}{\epsilon}  \cdot 3^{N_{0}(c)}$, and $E_{-}(c) = 9(1-2c)$. Thus, we complete the proof of this lemma.
	 	
	 	\end{proof}
	 \begin{lemma}\label{lemma17}
	 	For any $0\leq c <\frac{1}{3}$,  $O_{\epsilon}$ (the $\epsilon$-neighborhood of $S_{inv}$ defined above),~$\delta_{1}$ defined above, and any convex set~$\Omega_{0}$~in one of the nine regions with measure~$M(\Omega_{0})\geq \delta_{1}$, there exists a constant~$\beta_{2}>0$~depending on~$\epsilon$~and disjoint convexities~$\Omega_{i}\subset \Omega_{0} ,i=1,2,\cdots$~satisfying the following two conclusions:
	 	\begin{enumerate}
	 		\item for any~$i\geq 1$, there exists a~$t(i)$~such that~$T^{t(i)}(\Omega_{i})\subset O_{\epsilon}$;
	 		\item ~$M(\underset{i\geq 0}{\cup}\Omega_{i})\geq \beta_{2}M(\Omega_{0})$.
	 	\end{enumerate}
	 \end{lemma}
	
	 \begin{proof}
	 	As discussed in Lemma \ref{lemma16}, suppose any convex set and its iterative convex set do not intersect with ~$x_{1}=x_{2}$~and its pre-images, as shown in Fig. \ref{3xmod1frontimage}.
	 	
	 	If $\Omega_{0}$ is divided into at most three components after each iteration, there exists a component $\Omega_{1} \in T(\Omega_{0})$ such that
	 	$$
	 	M(\Omega_{1}) > \frac{1}{3} \cdot 9(1-2c) \cdot M(\Omega_{0}) = 3(1-2c) M(\Omega_{0}).
	 	$$
	 	Thus, we can construct a sequence of convex sets defined by $\Omega_{j+1} = T(\Omega_{j}) \cap~ \text{region}~i$, $i\in 1,2,\cdots,9 $ such that
	 	$$
	 	M(\Omega_{j+1}) > 3(1-2c) M(\Omega_{j}) > (3(1-2c))^{j} M(\Omega_{0}).
	 	$$
	 	Since \( M(\Omega_{0}) \geq \delta_{1} \), it follows that \( (3(1-2c))^{t} M(\Omega_{0}) > 1 \) for some fixed \( t = t(c) \).
	 	Thus, there exists some \( j < t \) such that \( T(\Omega_{j}) \cap ~\text{region} ~ i \neq \emptyset \) for every  \( i=1,2,\cdots 9 \).
	 	Consequently, by applying Lemma \ref{lemma15}, we can conclude that the results hold true.
	 	
	 	If $\Omega_{0}$ is first divided into four or five components after $a$ iterations, then $T^{a}(\Omega_{0})$ intersects with $x_{1} + x_{2} = 1$.
	 	Since \( M(\Omega_{0}) \geq \delta_{1} \), we have $|T^{a}(\Omega_{0}) \cap \{ x_{1} + x_{2} = 1 \}| \geq  \delta_{a} > 0$.
	 	Let $\Omega_{a}$ be the component in $T^{a}(\Omega_{0})$ that intersects with $x_{1} + x_{2} = 1$, then we have $|\Omega_{a} \cap \{ x_{1} + x_{2} = 1 \}| \geq  \delta_{a} > 0$.
	 Since   $\{x_{1} + x_{2} = 1\}$  in an invariant set for $T$, $\Omega_{a}$ continues to intersect with $x_{1} + x_{2} = 1$ after each iteration.
	 	Therefore, there exists a component $\Omega^{1}_{a} \in T(\Omega_{a})$ such that
	 	\begin{eqnarray*}
	 		|\Omega^{1}_{a} \cap \{ x_{1} + x_{2} =1 \}| &\geq & 3(1-2c)\cdot  |\Omega_{a} \cap \{ x_{1} + x_{2} = 1 \}| \\
	 		&\geq & 3(1-2c) \delta_{a}.
	 	\end{eqnarray*}
	 	Thus, we can construct a sequence of convex sets defined by $\Omega^{j+1}_{a} \subset T(\Omega^{j}_{a})$ such that
	 	\begin{eqnarray*}
	 		|\Omega^{j+1}_{a} \cap \{ x_{1} + x_{2} = 1 \}| &\geq & 3(1-2c)\cdot |\Omega^{j}_{a} \cap \{ x_{1} + x_{2} = 1 \}| \\
	 		&\geq & (3(1-2c))^{j} \delta_{a}.
	 	\end{eqnarray*}
	 	It follows that \( (3(2(1-2c))^{t} \delta_{a} > \frac{\sqrt{2}}{2} \) for some fixed \( t = t(c) \).
	 	Thus, there exists some \( j < t \) such that \( (\frac{1}{2}, \frac{1}{2}) \in \Omega^{j}_{a} \).
	 	Consequently, by applying Lemma \ref{lemma15}, we can conclude that the results hold true.
	 \end{proof}


	 \begin{lemma}\label{lemma18}
	 For any $0\leq c <\frac{1}{3}$, there exists a constant \( \delta_{2} \in (0,1) \) such that for almost every segment \( l \) in one of the nine regions with slope \( - 1 \) and measure \( M(l) < \delta_{2} \), there exist constants \( \delta_{3}, \beta_{3} > 0 \) and \( r_{0} \) depending on \( \delta_{2} \), as well as disjoint subsegments \( l_{i} \), \( i=1,2,\ldots \), such that the following conclusions hold:
	 	\begin{enumerate}
	 		\item For any \( i \geq 1 \), there exists a \( t(i) \geq 0 \) such that \( T^{t(i)}(l_{i}) \) is a component of \( T^{t(i)}(l) \) with \( M(T^{t(i)}(l_{i})) \geq \delta_{3} \) or \( T^{j}(l_{i}) \subset D_{r_{0}} \) for some \( j \leq t(i) \);
	 		\item \( M\left(\bigcup_{i} l_{i}\right) \geq \beta_{3} M(l) \).
	 	\end{enumerate}
	 \end{lemma}
	
	 \begin{proof}
	 	Let \( O_{r} = \{\bm{x} \in [0,1]^{2} \mid \text{dist}(\bm{x}, S_{inv}) \leq r\} \) for \( r > 0 \), and denote \( \overline{O}_{r} = \{\bm{x} \in O_{r} \mid \text{dist}(\bm{x}, \{x_{1} + x_{2} = 2/3\}) \leq r\} \);
	 	\(\widehat{O}_{r} = \{\bm{x} \in O_{r} \mid \text{dist}(\bm{x}, \{x_{1} + x_{2} = 4/3\}) \leq r\} \).
	 	Consider the segment \( l \subset O_{r_{0}} \)  with length less than \( \delta_{2} \).
	 	We fix \( c \) such that \( 3|1-2c| > 1 \), let \( M = \lfloor \log_{3(1-2c)}3\rfloor + 1 \), and choose \( \delta_{2} \) sufficiently small and \( r_{0} \) satisfying \( (3(1-2c))^{i} \delta_{2} <2r_{0} <3^{-i} \) for \( i =1, 2, 3, \ldots, M \).
	 	For a segment \( l \subset O_{r_{0}} \), we claim that \( T^{M}(l) \) has at most three components in $O_{r_{0}} $.
	 	
	 	Assuming \( T(l) \) has three components  leads to \( T(l) \subset  \overline{O}_{r_{0}}  \) or  \( T(l) \subset  \widehat{O}_{r_{0}} \).
	 	Without loss of generality, we only consider the former.
	 	Since \( \left(\frac{1}{3}, \frac{1}{3}\right) \in \overline{O}_{r_{0}} \) and the diameter of \( \overline{O}_{r_{0}} \) is less than \( 2r_{0} \), we have that \( T(l) \) is in the \(r_{0} \)-neighborhood of \( \left(\frac{1}{3}, \frac{1}{3}\right) \).
	 	Thus, we find that \( T^{i}(l_{i}) \) is in the \((3^{i-1}r_{0} )\)-neighborhood of \( T^{i-1}\left(\frac{1}{3}, \frac{1}{3}\right) \) for \( 2 \leq i \leq M \).
	 	Noting that  \( T^{j}\left(\frac{1}{3}, \frac{1}{3}\right) = (0,0)~ \text{or}~ (1,1)\) for \( j \geq 1 \) and $3^{i-1}r_{0}<\frac{1}{6}$, we conclude that the \((3^{i-1}r_{0}) \)-neighborhood of \( (0,0) \) or $(1,1)$ and \( \overline{O}_{r_{0}} \) do not intersect, thus proving our claim.
	 	
	 	Since \( (3(1-3c))^{M} > 3 \), we apply Iteration Lemma \ref{Lem:diedai} and Coro. \ref{Coro:diedai} with \( m_{0} = M \), \( a =3\), and \( E_{-}(c) = 3(1-2c) \) to complete the proof of this lemma.
	 \end{proof}

	  \begin{lemma}\label{lemma19}
	 	For any~$0 \leq c < \frac{1}{3}$,~$\delta_3 >0$ stated above, and~$r_0 = \frac{\delta_3}{3}$, consider a segment~$l_0$ in one of the four regions with slope~$-1$ and $M(l_{0})\geq \delta_3 $.
	 	Then there exists a   subsegments~$l^{'}_0$ of ~$l_0$  such that  $l^{'}_0\subset D_{r_0}$ and $M(l^{'}_0) \geq \frac{1}{3} M(l_0)$.
	 	\end{lemma}
	 	
	 	\begin{proof}
	 		If $l_{0}\cap O_{r_{0}} = \emptyset$, the conclusion naturally follows.
	 		
	 		If $l_{0}\cap O_{r_{0}} \neq \emptyset$,
	 		since the segment~$l_0$ has slope~$-1$ and~$M(l_0)\geq  \delta_3 =3r_0$, the maximum length of the intersection of $l_{0}$ and $O_{r_{0}}$ is $2r_{0}$, then in this case the length of $l_{0}$ outside $O_{r_{0}}$ is larger than
	 		$$
	 		\frac{M(l_{0})-2r_{0}}{M(l_{0})}=1- \frac{2r_{0}}{M(l_{0})}\geq 1-\frac{2r_{0}}{3r_{0}}=\frac{1}{3}.
	 		$$
	 		that is $M(l^{'}_0) \geq \frac{1}{3} M(l_0)$.
	 	\end{proof}

	Similar to the proof of Sec. \ref{sec:3},
	 Lemma  \ref{lemma15}, \ref{lemma16}, \ref{lemma17} and Lemma \ref{lemma18},  \ref{lemma19} can derive Lemma \ref{lemma6} and Lemma \ref{lemma8} respectively, while Lemma \ref{lemma6} and Lemma \ref{lemma8} can derive the ordered part and the disordered part of intermittent-synchronization respectively.

 \subsection{Proof of the uniqueness of absolutely continuous invariant measure when  \( c \in [0, \frac{1}{3}) \).}

  We demonstrate in detail the uniqueness of the absolute continuous invariant measure for CML of $f(x)=3x \mod 1$, and the relevant proof for $f(x)=-3x \mod 1$ is completely similar.

	\begin{lemma}\label{lemma20}
		Let the origin point be \( O = (0, 0) \),  the point \( M = \left( \frac{1}{6}, \frac{1}{6} \right) \) and $\epsilon>0$ depending on $T$.
		When \( c \in [0, \frac{1}{3}) \), for any open set \( G \), there exists an integer \( n \) such that \( T^n G \) contains a neighborhood of \( M \).
	\end{lemma}
	
	\begin{proof}
			By Lemma \ref{lemma16} and Lemma \ref{lemma17}, we know that for any open set \( G \),  there exists a component that intersects the diagonal  \( x_1 = x_2 \) after iterations.
		We denote the convex set of this component after iterations as \( \widehat{G} \) such that \( \widehat{G} \) intersects with the diagonal.
		
		Before \( \widehat{G} \) intersects the points \( (1/3, 1/3) \) or \( (2/3, 2/3) \) under iterations, the length of its intersection with the set \( \{x_1 = x_2\} \) increases by a factor of \( 3 \) with each iteration, it follows that \( \widehat{G} \) will eventually intersect  \( (1/3, 1/3) \) or  \( (2/3, 2/3) \) under iterations.
	   We only discuss the case where \( \widehat{G} \)  intersects with  \( (1/3, 1/3) \)  here, and the other case is completely the same.
		
		Let \( i_0 \in \mathbb{N} \) be the number of the iterations such that \( (1/3, 1/3) \notin T^{n}(\widehat{G}) \), $n=1,2, \cdots, i_{0}-1$ and \( (1/3, 1/3) \in T^{i_0}(\widehat{G}) \), then we can define an open set \( G_0 \subset T^{i_0}(\widehat{G}) \) such that  \( (1/3, 1/3) \in G_0 \).
		Let \( l_0 = |G_0 \cap \{x_1 = x_2\}| > 0 \), and define \( G_1 = G_0 \cap \text{region 5} \) and  \(l_{1}= |G_1 \cap \{x_1 = x_2\}|>0  \).
		Then $O=(0,0)\in TG_{1}$ and $ |TG_1 \cap \{x_1 = x_2\}|=3l_{1} $, as shown in Fig. \ref{unique2}.
		
		If \( M \in TG_1 \), then the conclusion holds directly.
		Therefore, we only need to consider  \( M \notin TG_1 \).
		In this case, we have
		\[
		|T^{2}G_1 \cap \{x_1 = x_2\}|=3^{2}l_{0} \quad \text{and} \quad O=(0,0)\in T^{2}G_{1}
		\]
		\begin{figure}
			\centering
			\includegraphics[width=0.5\textwidth]{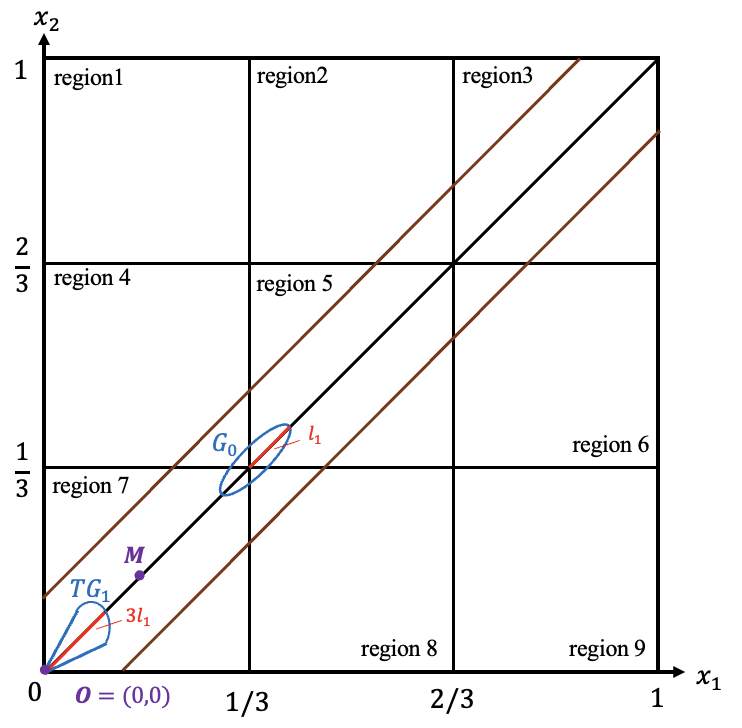}
			\caption{The figure shows that $(1/3, 1/3) \in  G_0 $, $G_1 = G_0 \cap~ \text{region} ~5$ and \( l_{1}=|G_1 \cap \{x_1 = x_2\}|  \).}
			\label{unique2}
		\end{figure}
		Since $O$ is a fixed point for $T$, then for any $i\in \mathbb{N}$, we have
		\[
		|T^{i}G_1 \cap \{x_1 = x_2\}|=3^{i}l_{0} \quad \text{and} \quad O=(0,0)\in T^{i}G_{1}
		\]
		Therefore, there exist $j\in \mathbb{N}$ such that
		\[
		|T^{j}G_1 \cap \{x_1 = x_2\}|=3^{j}l_{0} >|OM|
		\]
		which implies \( M \in T^{j}G_1 \), and the conclusion follows.
	\end{proof}
	
Note that for $0\le c<1/3$, each eigenvalue of $T$ possesses an absolute value strictly larger than 1 at every smooth point.
		From Theorem 3 in Tsujii \cite{tsujii2001absolutely}, we conclude that for CML with Lorenz map $f(x)= 3x ~\text{mod} ~1$ there exist finitely many absolutely continuous ergodic probability measures \( \mu_i \), where \( 1 \leq i \leq p \), and the basin of each measure \( \mu_i \), defined as
		\[
		\text{Basin}(\mu_i) = \left\{ x \in [0,1]^2 \mid \frac{1}{n} \sum_{j=0}^{n-1} \delta_{T^j(x)} \xrightarrow{w} \mu_i \right\},
		\]
		is an open set modulo sets of Lebesgue measure zero, and the union \( \bigcup_{i=1}^p \text{Basin}(\mu_i) \) has full Lebesgue measure in \( [0,1]^2 \).
		
		For two distinct absolutely continuous ergodic measures \( \mu_1 \) and \( \mu_2 \), we define two small balls \( B_1 \subset \text{Basin}(\mu_1) \) and \( B_2 \subset \text{Basin}(\mu_2) \).
		From Lemma \ref{lemma20}  we know that after iterations, both \( B_1 \) and \( B_2 \) will contain neighborhoods of \( M \), denoted as \( M_{\epsilon_1} \) and \( M_{\epsilon_2} \) respectively.
		Therefore, both \( B_1 \) and \( B_2 \) will contain a common neighborhood of \( M \)  after iterations, specifically \( M_{\epsilon_1} \cap M_{\epsilon_2} \).
		This shows that basin of different absolutely continuous ergodic measure $\mu_{i}$ will contain the same open set modulo a Lebesgue null set after iterations.
		Therefore, there can be only one absolutely continuous ergodic probability measure,  thereby proving the uniqueness of the measure.

	 \section{\label{sec:5}Analysis and application of geometric-combinatorics Method}
	 In this section, we discuss the potential applications of our method to different kinds of CMLs and compare our new geometric-combinatorics method with the traditional Frobenius-Perron operator-based method.
	
	 \subsection{multi-node nearest-neighbor coupled map lattice}
	 The nerest-neighbor coupled map lattice is a locally coupled system in which each node interacts only with its neighbors.
	 The case of the 8-node nearest-neighbor coupling is shown in panel (a) of Fig.  \ref{Fig.8nearest6global} as an example.
	
	  \begin{figure}
	 	\includegraphics[width=0.48\textwidth]{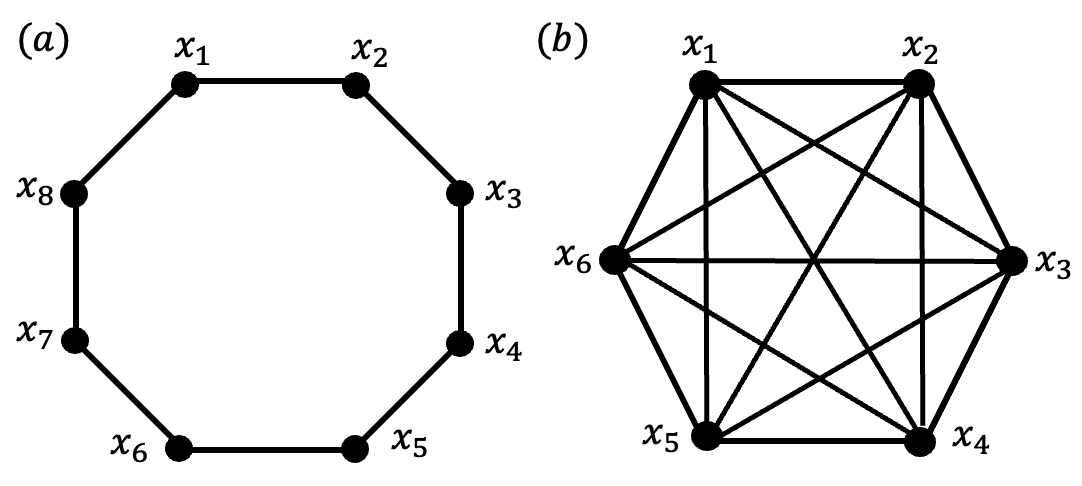}
	 	\caption{\label{Fig.8nearest6global} The figure (a) shows a schematic diagram of 8-node nearest-neighbor coupling and figure (b) shows a schematic diagram of 6-node globally coupling.
	 Each	$x_{i}$ represents the state of node $i$.}
	 \end{figure}
	
	 For an $n$-node nearest-neighbor coupled map lattice, the dynamical equation for each node is given by:
	
	 \begin{eqnarray}\label{nearest}
	 	T: \left\{
	 		\begin{array}{ll}
	 			\bar{x}_1 &= (1 - 2c) f(x_1) + c f(x_2)  + c f(x_n),  \\
	 		\bar{x}_2 &= c f(x_1) + (1 - 2c) f(x_2) + c f(x_3), \\
	 		&\vdots \\
	 		\bar{x}_n &= c f(x_1) + c f(x_{n-1}) +  (1 - 2c) f(x_n).
	 		\end{array}
	 		\right.
	 \end{eqnarray}
	 where $\bar{x}_i$ represents the next state of node $i$,   $f(x_i)$ is the standard Lorenz map for node $i$
	 and $c$ is the coupling strength which is in the range $(0,0.5)$.
	
	 Since the slope of the Lorenz map is constantly 2, the Jacobian matrix $J$ of (\ref{nearest}) takes the form:
	
	 \[
	 J = 2 \begin{pmatrix}
	 	(1 - 2c) & c &0& \dots & c \\
	 	c & (1 - 2c) &c& \dots & 0 \\
	 	\vdots & \vdots & \ddots & \vdots \\
	 	c & 0 & 0&\dots & (1 - 2c)
	 \end{pmatrix}.
	 \]
	
	 The eigenvalues of $J$ are given by
	 \[
	 \lambda_k = 2(1 - 2c + 2c\cos(\frac{2k\pi}{n})), \quad k=0,1,\dots, n-1.
	 \]
	 Among these eigenvalues, the maximum eigenvalue is $\lambda_0 = 2$, and the minimum eigenvalue satisfies $\lambda_{n-1} > 2(1 - 4c)$.
	
	 \begin{figure}
	 	\includegraphics[width=0.45\textwidth]{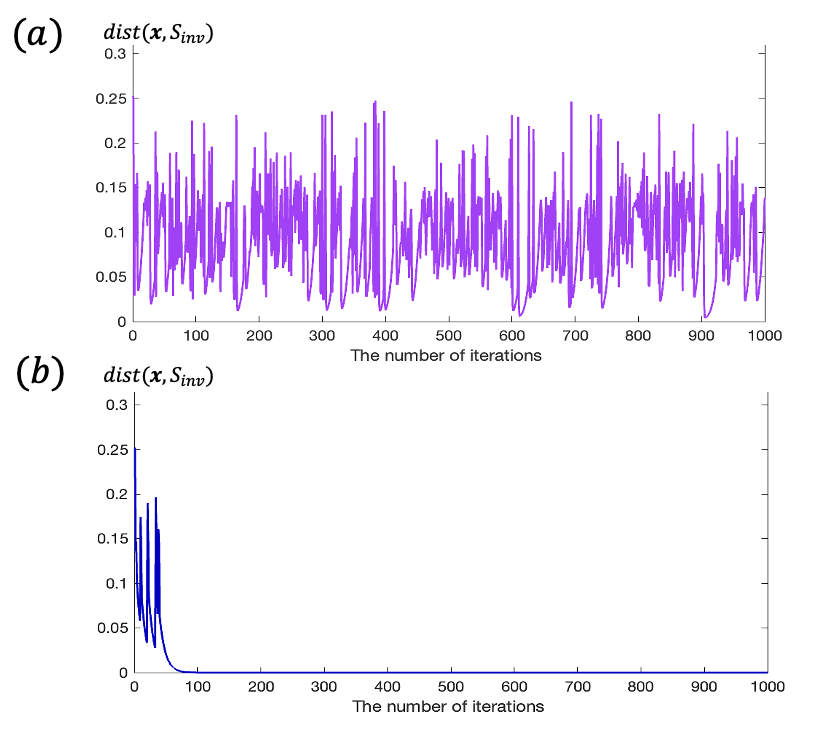}
	 	\caption{\label{Fig.5node} The figures show the change of the distance of any point in the phase space $[0,1]^{5}$ to the invariant set $S_{inv}=\{\bm{x}\in[0,1]^{5}| x_{1}=x_{2}=\cdots =x_{5}\}$ after iterations in the 5-node nearest-neighbor CML.
	 	In penal (a), $c=0.3$, and in panel (b), $c=0.4$.}
	 \end{figure}
	
	 	 \begin{figure}
	 	\includegraphics[width=0.45\textwidth]{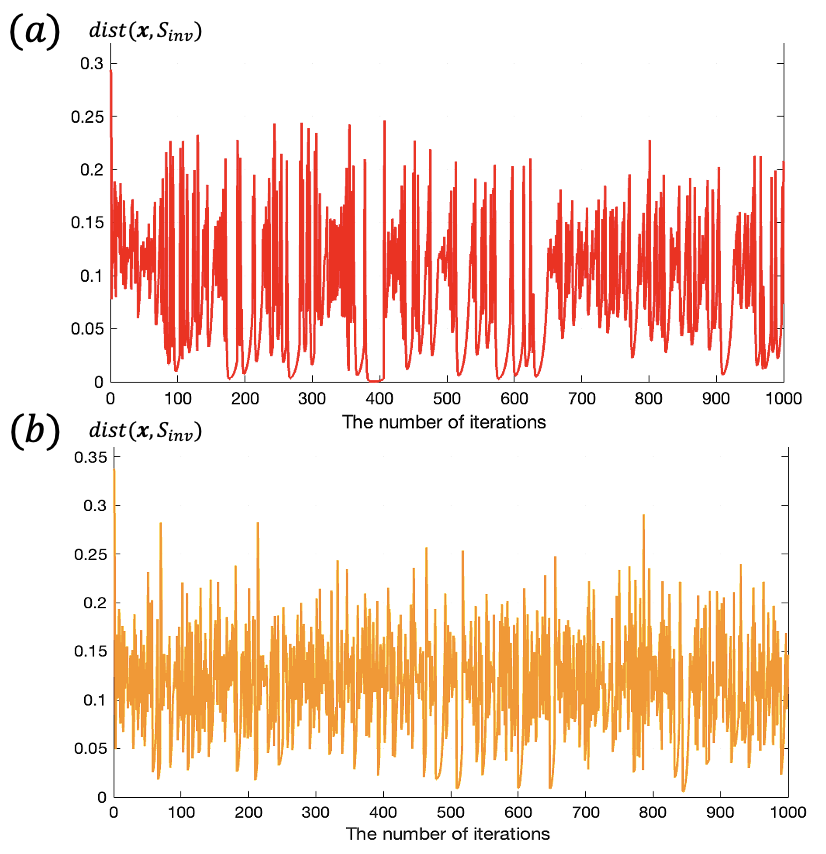}
	 	\caption{\label{Fig.67node} The figures show the change of the distance of any point in the phase space $[0,1]^{n}$ to the invariant set $S_{inv}=\{\bm{x}\in[0,1]^{n}| x_{1}=x_{2}=\cdots =x_{n}\}$ after iterations in the n-node nearest-neighbor CML.
	 		In penal (a), $n=6$, $c=0.4$, and in panel (b), $n=7$, $c=0.4$.}
	 \end{figure}
	 For $n >5$, regardless of the choice of $c$ within the range $(0,0.5)$, we always have $\lambda_1 > 1$. This implies that, besides expansion along the vector direction $(1,1,\cdots 1)$,  the transformation $T$ remains expansive in other directions with iterations.
	 Consequently, complete synchronization cannot occur.
	
	 \begin{figure}
	 	\includegraphics[width=0.43\textwidth]{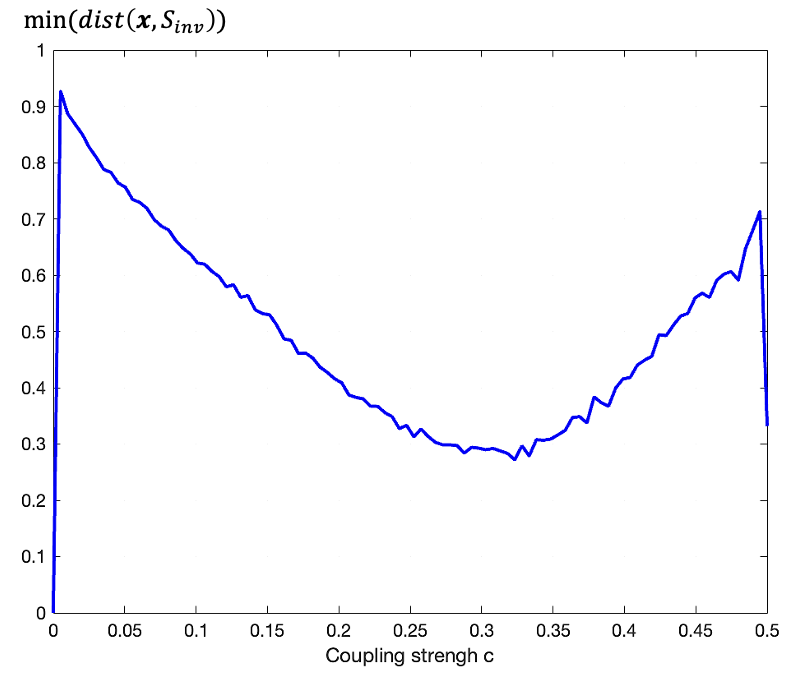}
	 	\caption{\label{Fig.6node} The figure shows the change of the average distance of any point in the phase space $[0,1]^{6}$ to the invariant set $S_{inv}=\{\bm{x}\in[0,1]^{6}| x_{1}=x_{2}=\cdots =x_{6}\}$ of the 80000th to 81000th iteration in the 6-node nearest-neighbor CML.}
	 \end{figure}
	
Through numerical simulation, we also find that when the number of nodes $n\leq 5$, synchronization and intermittent-synchronization alternately occur with the change of coupling strength $c$.
For 5-node CML, intermittent-synchronization occurs when $c=0.3$, as shown in panel (a) of Fig. \ref{Fig.5node}, and synchronization occurs when $c=0.4$, as shown in panel (b) of Fig. \ref{Fig.5node}.
However, when the number of nodes $n>5$, the numerical simulation shows that no matter how the coupling strength changes, only intermittent-synchronization occurs, for example, intermittent-synchronization occurs when $c=0.4$, $n=6$ or $7$, as shown in Fig. \ref{Fig.67node}, and Fig. \ref{Fig.6node} shows that no matter what the value of $c$ is, the average distance from the point to $S_{inv}$ after iterations can not approach 0.
These numerical simulation results are in complete agreement with our theoretical analysis.
	
	 Therefore, studying   transitions in multi-node nearest-neighbor coupled map lattices is meaningless.
	 We anticipate that our geometric-combinatorics method will provide a more detailed analysis of intermittent-synchronization in multi-node nearest-neighbor coupled map lattices in future research.
	
	 \subsection{multi-node globally coupled map lattice}
	 The globally coupled map lattice is a globally interactive system in which each node interacts with all other nodes.
	 Therefore, each node is affected not only by itself, but also by the average state of the entire system.
	 The case of the 6-node globally coupling is shown in panel (b) of Fig. \ref{Fig.8nearest6global} as an example.
	
	 For an $n$-node globally coupled map lattice, the dynamical equation for each node is given by:
	 	 \begin{eqnarray}\label{global}
	 	T:	\left\{
	 		\begin{array}{ll}
	 			\bar{x}_1 &= (1 - (n - 1)c) f(x_1) + c f(x_2)  + \dots + c f(x_n), \\
	 		\bar{x}_2 &= c f(x_1) + (1 - (n - 1)c) f(x_2) + \dots + c f(x_n), \\
	 		&\vdots \\
	 		\bar{x}_n &= c f(x_1) + c f(x_2) + \dots + (1 - (n - 1)c) f(x_n),
	 		\end{array}
	 		\right.
	 \end{eqnarray}
	 where $\bar{x}_i$ represents the next state of node $i$,
 $f(x_i)$ is the standard Lorenz map for node $i$,
 and $c$ is the coupling strength.

	 Each node's dynamics are determined by its own state and the states of all other nodes, with the coupling strength $c$ determining the degree of mutual influence between nodes.
	
	 The elements of the Jacobian matrix $J$ are given by $\frac{\partial \bar{x}_i}{\partial x_j}$. Since the slope of the Lorenz map is constantly 2, the Jacobian matrix $J$ of (\ref{global}) takes the form:
	
	 \[
	 J = 2 \begin{pmatrix}
	 	(1 - (n - 1)c) & c & \dots & c \\
	 	c & (1 - (n - 1)c) & \dots & c \\
	 	\vdots & \vdots & \ddots & \vdots \\
	 	c & c & \dots & (1 - (n - 1)c)
	 \end{pmatrix}.
	 \]
	
	 To determine the stability of the synchronized solution, we analyze the eigenvalues of the Jacobian matrix:
	 \begin{itemize}
	 	\item In the direction of the vector $(1,1,\cdots,1)$, the Jacobian matrix has the eigenvalue:
	 	\[
	 	\lambda_1 = 2.
	 	\]
	 	\item For the remaining $n-1$ eigenvalues, we obtain:
	 	\[
	 	\lambda_2 = \lambda_3 = \dots = \lambda_n = 2 (1 - nc).
	 	\]
	 \end{itemize}
	
	 When $c < \frac{1}{2n}$, we have $\lambda_2 = \lambda_3 = \dots = \lambda_n > 1$, indicating that the system experiences strong perturbations in all directions except the vector direction  $(1,1,\cdots,1)$.
	 As a result, the system cannot return to the synchronized state, since perturbations in all directions drive the system away from the synchronized solution.
	
	 As the number of nodes $n$ increases, the range of $c$ satisfying this condition continuously shrinks.
	 This implies that the maximum coupling strength for which intermittent-synchronization occurs decreases towards zero as the number of nodes increases.
	 Consequently, studying intermittent-synchronization in multi-node globally coupled map lattices holds little significance.
	
	 \subsection{Results Obtained Using geometric-combinatorics  Method}
	 By applying our geometric-combinatorics method to the two-node coupled non-standard tent map lattice, we have determined the range of coupling strengths corresponding to the occurrence of intermittent-synchronization when the tent map’s vertex lies within the \( \epsilon \)-neighborhood of \( (\frac{1}{2},1) \) with $\epsilon =0.01$.
	 Additionally, we have identified the critical coupling strength at which the transition between synchronization and intermittent-synchronization occurs.
	 These results cannot be obtained using the traditional Forbenius-Perron operator-based method.
	 We have compiled these results and are ready to submit them.
	
	 \subsection{Comparison Between the Geometric-Combinatorics and Forbenius-Perron Operator-Based Methods}
	 Our method provides a more intuitive mathematical description of both synchronization and intermittent-synchronization.
	 Moreover, the geometric-combinatorics method enables us to analyze the full range of coupling strengths and observe the corresponding dynamical behaviors, which is beyond the capability of the traditional Forbenius-Perron operator-based method.
	
	 For example, in the case of a two-node coupled standard tent map lattice, Keller\cite{keller1997mixing}, using the Forbenius-Perron operator, was only able to analyze the dynamical system under sufficiently small coupling strengths and could not capture the transition between different dynamical behaviors.
	 In contrast, Li et al.\cite{li2021intermittent} employing the geometric-combinatorics method, successfully identified the critical coupling strength marking the transition between intermittent-synchronization and complete synchronization.
	
	 Here is another example: For the $n$-node coupled tent map lattice, Keller\cite{keller1997mixing} used the traditional Frobenius-Perron operator method to discuss the dynamical system phenomena only for the coupling strength $c \in (0, c_{\text{uni}})$, where $c_{\text{uni}}$ is sufficiently small and does not provide any estimates.
	 In contrast, Li et al.\cite{li2021intermittent}, using the geometric-combinatorics method, found that intermittent-synchronization phenomena occurs within $c \in (0, c_*)$, where $c_*$ is related to the number of nodes $n$.
	 Meanwhile, by applying the iterative lemma, we can estimate the corresponding coupling strength range -- that is, intermittent-synchronization occurs in the $n$-node coupled tent map lattice when
	 \[
	 c < \frac{1}{2} - \frac{(2^n - 1)^{1/n}}{4}.
	 \]

	 \section{\label{sec:6}Conclusions}
	 In this paper, we develop the new geometric-combinatorics method originated from \cite{li2021intermittent}to study the dynamic behaviors of coupled map lattice (CML) with discontinuous piecewise expanding maps including $f(x)= 2x~ \text{mod} ~ 1 ~~~\text{or}~~~~ f(x)=\pm 3x ~ \text{mod} ~  1$ on $[0,1]$. The discontinuity of the Lorenz map complicates the study.
	 We proved the existence of a threshold on the coupling strength for the occurrence of intermittent-synchronization and the uniqueness of absolutely continuous invariant measure for CML, which coincides with numerical results.
	
In practice, the uniqueness (and existence) of absolutely continuous invariant measure obtained in this paper together with
Tsujii's result \cite{tsujii2001absolutely} shows that the trajectory of almost each point in the sense of Lebesgue measure will visit almost everywhere of a region with a positive Lebesgue measure for infinite times.

Based on the new geometric-combinatorics method, we have obtained the synchronization--intermittent synchronization transition and the uniqueness of ACIM for various situations, which is not limited to the case with $2$ nodes.
	 Further work on this avenue of investigation is ongoing.
	
		 In comparison, most similar results obtained using  Forbenius-Perron operator \cite{keller1996coupled,fischer2000transfer,rugh2002coupled} depend heavily on the smallness condition  on the coupling strength with the exception for the case with $2$ nodes and standard (slope=$2$) Lorenz type map  by Keller et al. \cite{keller1992some} which seems difficult to be generalized directly to $3$ or more nodes. Moreover, the abstract method based on Frobenius-Perron operator may leads to partial loss of geometric information. For example, CML with identical maps is symmetric with respect to the synchronization manifold (the diagonal). It implies that the diagonal must play a key role in the dynamical behavior of CMLs, which is beyond the consideration of abstract functional sketch. Hence results based on Forbenius-Perron operator are far from numerical results. In fact, the information from the diagonal is necessary in the sufficient-necessary result of Keller et al \cite{keller1992some}. This implies the importance of systemically developing a  geometric method.

	

\appendix
\section{Proof of Lemma \ref{Lem:diedai} and Corollary \ref{Coro:diedai}}

The specific proof of Lemma \ref{Lem:diedai}  is as follows.
\begin{proof}
	We iterate the map on~$\Omega$~for~$k(N)$~times.
	Then there are at most~$a^{\lfloor \frac{k(N)}{m_{0}} \rfloor+1}$~disjoint sets~$\widehat{\Omega}_{j}\subset \Omega$~such that~$T^{k(N)}(\widehat{\Omega}_{j})$~is a component of~$T^{k(N)}(\Omega)$~for each~$j$~(note that the assumption that~$M(T^{k(N)}(\widehat{\Omega}_{j})) \leq \delta_{1}M(D)$~is valid for each~$\widehat{\Omega}_{j}$~and each~$i \leq k(N)$).
	
	Then the total measure of all~$\widehat{\Omega}_{j}$'s satisfying~$M(T^{k(N)}(\widehat{\Omega}_{j})) \leq 2^{-\mu^{N}}M(D)$~is less than

	\begin{widetext}
		\begin{eqnarray*}
			& 2^{-\mu^{N}}&M(D)a^{\lfloor \frac{k(N)}{m_{0}} \rfloor+1} E_{-}(c)^{-k(N)} \\
			& \leq& a 2^{-\mu^{N}}M(D) \left(\frac{E_{-}(c)}{a^{1/m_{0}}}\right)^{-k(N)} \quad   (\text{since } M(\widehat{\Omega}_{j}) \leq E_{-}(c)^{-k(N)}M(T^{k(N)}(\widehat{\Omega}_{j}))) \\
			&    \leq &a 2^{-\mu^{N}}M(D) \left(\frac{E_{-}(c)}{a^{1/m_{0}}}\right)^{\lfloor \log_{E_{+}(c)} \frac{M(\Omega)}{\delta_{1}M(D)} \rfloor + 1} \\
			&    \leq& a 2^{-\mu^{N}}M(D) \left(\frac{E_{-}(c)}{a^{1/m_{0}}}\right)^{1-\log_{E_{+}(c)}\delta_{1}} \left(\frac{E_{-}(c)}{a^{1/m_{0}}}\right)^{-\log_{E_{+}(c)}M(D)} \left(\frac{E_{-}(c)}{a^{1/m_{0}}}\right)^{-\log_{E_{+}(c)}M(\Omega)} \\
			& =& F(c)2^{-\mu^{N}}M(D)^{1-\log_{E_{+}(c)} \left(\frac{E_{-}(c)}{a^{1/m_{0}}}\right)} M(\Omega)^{\log_{E_{+}(c)} \left(\frac{E_{-}(c)}{a^{1/m_{0}}}\right)} \quad (\text{since } a^{\log_{b}c} = c^{\log_{b}a}).
		\end{eqnarray*}
	\end{widetext}
	Hence it possesses a portion of~$\Omega$~less than
	\begin{eqnarray*}
		&	F(c)2^{-\mu^{N}} \left(\frac{M(\Omega)}{M(D)}\right)^{\log_{E_{+}(c)} \left(\frac{E_{-}(c)}{a^{1/m_{0}}}\right) - 1} \\
		&\leq F(c)2^{-\mu^{N}} \left(2^{-\mu^{N+1}}\right)^{\log_{E_{+}(c)} \left(\frac{E_{-}(c)}{a^{1/m_{0}}}\right) - 1} \leq F(c)2^{-d\mu^{N}}.
	\end{eqnarray*}
	Choose~$\Omega_{1}, \cdots$~to be all~$\widehat{\Omega}_{j}$~with a measure larger than~$ 2^{-\mu^{N}}M(D)$~and the proof is completed.
\end{proof}

The specific proof of Corollary \ref{Coro:diedai} is as follows.
	\begin{proof}
		Let~$N$~be the unique integer such that~$2^{-\mu^{N+1}}M(D)\leq M(\Omega)<2^{-\mu^{N}}M(D)$~and denote~$\Omega_{N+1}=\Omega$.
		Applying the Iteration Lemma \ref{Lem:diedai} on~$\Omega_{N+1}$, there exist disjoint sets~$\Omega^{i}_{N+1} \subset \Omega_{N+1}$,~$i = 1, \cdots$, such that
		\begin{itemize}
			\item ~$T^{k(N)}(\Omega^{i}_{N+1})$~is a component of~$T^{k(N)}(\Omega_{N+1})$~with
			
			$M(T^{k(N)}(\Omega^{i}_{N+1})) \geq 2^{-\mu^{N}}M(D)$~for each~$i$~and
			\item ~$M\left(\sum_{i \geq 1} \Omega^{i}_{N+1}\right) \geq (1 - F(c)2^{-d\mu^{N}}) M(\Omega_{N+1})$.
		\end{itemize}
		
		Let
		$$
		I_{N} = \{ i \mid M(T^{k(N)}(\Omega^{i}_{N+1})) \in [2^{-\mu^{N}}M(D), 2^{-\mu^{N-1}}M(D)] \},
		$$
		and denote the set of all other~$i$~by~$I_{N}^{'}$.
		For~$i \in I_{N}$, applying the Iteration Lemma \ref{Lem:diedai} on~$T^{k(N)}(\Omega^{i}_{N+1})$, we obtain disjoint sub-curves or measurable subsets~$\Omega^{i,j}_{N} \subset \Omega^{i}_{N+1}$,~$j = 1, \cdots$~and~$k_{i}(N)$ such that
		\begin{itemize}
			\item ~$T^{k_{i}(N)}(\Omega^{i,j}_{N})$~is a component of~$T^{k_{i}(N)}(\Omega^{i}_{N+1})$~with~
			
			$M(T^{k_{i}(N)}(\Omega^{i,j}_{N})) \geq 2^{-\mu^{N-1}}M(D)$~for each~$j$~and
			\item ~$M\left(\sum_{j \geq 1} \Omega^{i,j}_{N}\right) \geq (1 - F(c)2^{-d\mu^{N-1}}) M(\Omega^{i}_{N+1})$.
		\end{itemize}
		
		Thus, we have
		\begin{eqnarray*}
			M\left( \cup_{i \in I_{N}} \left( \cup_{j} \Omega^{i,j}_{N} \right) \cup \left( \cup_{i \in I_{N}^{'}} \Omega^{i}_{N+1} \right) \right)\\
			\geq (1 - F(c)2^{-d\mu^{N}})(1 - F(c)2^{-d\mu^{N-1}}) M(\Omega_{N+1}).
		\end{eqnarray*}
		Moreover, the sets on the left-hand side of the above inequality are disjoint, and for each set~$\widetilde{\Omega}$, there exists a~$k(\widetilde{\Omega})$~such that~$T^{k(\widetilde{\Omega})}(\widetilde{\Omega})$~is a component with~$M(T^{k(\widetilde{\Omega})}(\widetilde{\Omega})) \geq 2^{-\mu^{N-1}}M(D)$.
		
		By induction, we obtain the existence of~$\Omega_{i}$~,~$i = 1, \cdots$~, such that (1) and (2) hold true, where~$c_{1} = \prod_{j=N_{0}}^{N} \left(1 - F(c) (2^{d})^{-\mu^{j}} \right) > 0$, which has a positive lower bound for all~$N$.
		
		In fact, it is sufficient to prove that
		$$
		\prod_{j=N_{0}}^{\infty} \left( 1 - F(c)(2^{d})^{-\mu^{j}} \right) > 0,
		$$
		which can be obtained from the fact that
		$$
		\sum_{j=N_{0}}^{\infty} \ln \left( 1 - F(c)(2^{d})^{-\mu^{j}} \right) \geq F(c) \sum_{j=N_{0}}^{\infty} (2d)^{-\mu^{j}} > -\infty.
		$$
		
		This completes the proof.
	\end{proof}

		\begin{acknowledgments}
		We express our acknowledgement to Jianyu Chen and Jiaohao Xu. We also appreciate the helpful comments from the anonymous reviewers.
		\end{acknowledgments}
		
		\section*{Data Availability Statement}
		
		Data sharing is not applicable to this article as no new data were created or analyzed in this study.
		\bibliography{Lorenzarxiv}

	\end{document}